%% file: ms.tex
\newcommand{\Cauchy}{\mathcal{C}}

\newcommand{\A}{\mathbb{A}}

\newcommand{\cb}{\mathbf{c}}

\newcommand{\X}{\mathbf{X}}
\newcommand{\Y}{\mathbf{Y}}

\newcommand{\DFT}{\mathbb{F}}
\newcommand{\f}{\mathbf{f}}
\newcommand{\g}{\mathbf{g}}

\newcommand{\x}{\mathbf{x}}
\newcommand{\y}{\mathbf{y}}

\newcommand{\Vanderm}{\mathbb{V}}

\newcommand{\ii}{\mathfrak{i}}
\newcommand{\ee}{\mathbf{e}}
\newcommand{\roff}{{\boldsymbol\varepsilon}}
\newcommand{\bfalpha}{{\boldsymbol\alpha}}
\newcommand{\bfbeta}{{\boldsymbol\beta}}

\newcommand\restr[2]{{
		\left.\kern-\nulldelimiterspace 
		#1 
		\vphantom{\big|} 
		\right|_{#2} 
	}}
\documentclass[final,10pt,preprint]{aimdyn-rep}

\renewcommand\reportyear{2018}  
\renewcommand\reportmonthcreated{April}	
\renewcommand\reportnumber{001} 
\renewcommand\thisversion{1}  
\renewcommand\revisiondate{\today}
\renewcommand\thiscontract{DARPA HR0011-16-C-0116}  
\allowdisplaybreaks

\usepackage{algorithmic}
\usepackage{algorithm}
\usepackage[]{mcode}
\usepackage{float}
\usepackage{relsize}
\usepackage{tikz-cd}

\usepackage{tikz}
\usepackage{verbatim}
\usepackage[margin=1in]{geometry}
\usepackage{subcaption}  
\usepackage{enumerate}
\usepackage{paralist}
\usepackage{xcolor} 
\usepackage[hypcap]{caption}
\usepackage{url}
\usepackage{mathrsfs}  
\usepackage{mathtools}
\usepackage{bm} 	
\usepackage{bbm} 	

\newcommand{\mc}{\mathcal}

\newcommand{\mbf}{\bm} 			
\newcommand{\mbb}{\mathbb}  	

\newcommand{\N}{ \ensuremath{\mathbb{N}}}  

\newcommand{\C}{ \ensuremath{\mathbb{C}}}  

\newcommand{\eps}{\epsilon}

\providecommand\given{} 
\newcommand\SetSymbol[1][]{ \nonscript\,#1\vert \allowbreak \nonscript\,\mathopen{} }
\DeclarePairedDelimiterX\set[1]{\lbrace}{\rbrace}{ \renewcommand\given{\SetSymbol[\delimsize]} #1 }  
\DeclareMathOperator*{\argmin}{{arg\, min}}

\DeclarePairedDelimiterX\norm[1]{\lVert}{\rVert}{#1}  			
\DeclarePairedDelimiterX\inner[2]{\langle}{\rangle}{#1 \,,\, #2}  	
\DeclareMathOperator{\linspan}{span} 							
\DeclareMathOperator{\Ran}{Ran}											

\DeclarePairedDelimiterX\abs[1]{\lvert}{\rvert}{#1} 

\newcommand{\0}{\mathbf{0}}

\begin{document}
	

	\title[Vandermonde-Cauchy DMD]{Data driven Koopman spectral analysis in Vandermonde--Cauchy form via the DFT: numerical method and theoretical insights}
	
	
	\author[Drma\v{c}, Mezi\'{c} and Mohr]{Zlatko Drma\v{c}\affil{1}\comma\corrauth and Igor Mezi\'{c}\affil{2,3} and Ryan Mohr\affil{3}}
	\address{\affilnum{1}\ Faculty of Science, Department of Mathematics, University of Zagreb, 10000 Zagreb, Croatia.\\
	\affilnum{2}\ Department of Mechanical Engineering and Mathematics, University of California, Santa Barbara, CA
	93106, USA \\
	\affilnum{3}\ AIMdyn, Inc., Santa Barbara, CA 93101, USA}
	%
	%
	\emails{{\tt drmac@math.hr} (Z. Drma\v{c}), {\tt mezic@engr.ucsb.edu} (I. Mezi\'{c}), {\tt mohrr@aimdyn.com} (R. Mohr)}
	%
	
\begin{abstract}
The goals and contributions of this paper are twofold. It provides a new computational tool for data driven Koopman spectral analysis by taking up the formidable challenge to develop a numerically robust algorithm by following the natural formulation via the Krylov decomposition with the Frobenius companion matrix, and by using its eigenvectors explicitly -- these are defined as the inverse of the notoriously ill--conditioned  Vandermonde matrix. The key step to curb ill--conditioning is the discrete Fourier transform of the snapshots; in the new representation, the Vandermonde matrix is transformed into a generalized Cauchy matrix, which then allows accurate computation by specially tailored algorithms of numerical linear algebra. 
The second goal is to shed light on the connection between the formulas for optimal reconstruction weights when reconstructing snapshots using subsets of the computed Koopman modes. It is shown how using a certain weaker form of generalized inverses leads to explicit reconstruction formulas that match the abstract results from Koopman spectral theory, in particular the Generalized Laplace Analysis. 
\end{abstract}
	
	
	
	\maketitle
	
\tableofcontents

\input{SOURCES/intro}
\input{SOURCES/Preliminaries}
\input{SOURCES/Section_Vanderm_DFT}

\input{SOURCES/Section_Numerical_Example_Case-Study}

\input{SOURCES/Section_Freq_Domain_Reconstruct}

\input{SOURCES/gla-connection}

\input{SOURCES/conclusions}

\bibliography{SOURCES/Reconstruction_references}
\bibliographystyle{plain}

\input{SOURCES/Section_Appendix_1}
\end{document}

%% file: SOURCES/intro.tex
\section{Introduction}
Dynamic Mode Decomposition is a data driven spectral analysis technique for a time series. For a sequence of snapshot vectors $\f_1,\f_2,\dots, \f_{m+1}$ in $\mathbb{C}^{n}$, assumed driven by a linear operator $\A$, $\f_{i+1}=\mbb A \f_i$, the goal is to represent the snapshots in terms of the computed eigenvectors and eigenvalues of $\A$. Such a representation of the data sequence provides an insight into the evolution of the underlying dynamics, in particular on dynamically relevant spatial structures (eigenvectors) and amplitudes and frequencies of their evolution (encoded in the corresponding eigenvalues) -- it can be considered a finite dimensional realization of the Koopman spectral analysis, corresponding to the Koopman operator associated with the dynamics under study \cite{arbabi-mezic-siam-2017}. This important theoretical connection with the Koopman operator and the ergodic theory, and the availability of numerical algorithm \cite{schmid2010} make the DMD a tool of trade in computational study of complex phenomena in fluid dynamics, see e.g. \cite{rowley2009spectral}, \cite{Williams2015}. A peculiarity of the data driven setting is that  direct access to the operator is not available, thus an approximate representation of $\mbb A$ is achieved using solely the snapshot vectors whose number $m$ is usually much smaller than the dimension $n$ of the domain of $\A$.

In this paper, we revisit the natural formulation of finite dimensional Koopman spectral analysis in terms of Krylov bases and the Frobenius companion matrix. In this formulation, reviewed in \S \ref{SS=Modal-appr-snapshots} below, a spectral approximation of $\A$ is obtained by the Rayleigh--Ritz extraction using the \emph{primitive} Krylov basis $\X_m=(\f_1,\ldots,\f_m)$. This means that the Rayleigh quotient of $\A$ is the Frobenius companion matrix, whose eigenvector matrix is the inverse of the Vandermonde matrix $\Vanderm_m$, parametrized by its eigenvalues $\lambda_i$. The DMD (Koopman) modes, that is, the Ritz vectors $z_i$, are then computed as $Z_m\equiv (z_1\;\ldots\;z_m)=\X_m\Vanderm_m^{-1}$. Then, from $\X_m = Z_m \Vanderm_m$ we readily have
$\f_i = \sum_{j=1}^m z_j \lambda_j^i$ for $i=1,\ldots,m$. 

Unfortunately, Vandermonde matrices can have extremely high conditions numbers, so a straight-forward implementation of this algebraically elegant scheme can lead to inaccurate results in finite precision computation. Most other DMD-variants bypass this issue by computing an orthonormal basis from the snaphshots using a truncated singular value decomposition \cite{schmid2010,Schmid2011,Tu-DMD-Theory-Appl}. We note that the original formulation of DMD was based on the SVD (\cite{schmid2010}, reviewed in Algorithm \ref{zd:ALG:DMD} in this paper), whereas the first connection between DMD and Koopman operator theory was formulated in terms of the companion matrix \cite{Rowley:2009ez}. In Rowley et al. \cite{Rowley:2009ez}, though, there was no consideration of the deep numerical issues related to working with Vandermonde matrices which has likely contributed to the prevalence of the SVD-based variants of DMD. There is, however, a certain intrinsic elegance in the decomposition of the snapshots in terms of the spectral structure of the companion matrix in addition to a stronger connection to Generalized Laplace Analysis (GLA) theory \cite{Mohr:2014wm,Mezic:2013ei} of Koopman operator theory than the SVD-based DMD variants have.

Following this natural formulation, we present a DMD algorithm capable of numerically robust computing without resorting to the singular value decomposition of $\X_m$. We do this by leveraging high accuracy numerical linear algebra techniques to curb the potential ill-conditioning of the Vandermonde matrix. The key step is to transform the snapshot matrix by the discrete Fourier transform (DFT) -- this induces a similarity of the companion matrix, and transforms its eigenvector matrix into the inverse of a generalized Cauchy matrix,  which then allows accurate computation by specially tailored algorithms. This will be achieved using the techniques introduced in \cite{dgesvd-99}, \cite{Demmel-99-AccurateSVD}. We present the details of our formulation in section \ref{zd:SS=DFT+Cauchy-solver}. Numerical example presented in section \ref{SS:reconstruct:Num-Example}, where the spectral condition number of $\Vanderm_m$ is above $10^{70}$, illustrates the potential of the proposed method.

While the full collection of DMD modes and eigenvalues gives insight into the intrinsic physics, in many applications a reduced-order model is required. Thus reconstructions of the snapshots using a subset of the modes (selected using some criteria that we do not address here) is desired. We take up the subject of optimal reconstruction of the snapshots in section \ref{S=Snapshot-rec-theory}. The problem of computing optimal reconstruction weights is formulated in terms of reflexive g-inverses (see \cite{rao-mitra-1972}), which are somewhat weaker than the more well-known Moore-Penrose pseudo-inverse. It is this set of results that is closely linked with Generalized Laplace Analysis (GLA). We compare formulas for the optimal reconstruction weights, derived in section \ref{S=Snapshot-rec-theory}, with reconstruction formulas given by GLA theory in section \ref{S=GLA-connection}. Whereas the GLA reconstruction formulas have the form of a (finite) ergodic average, the optimal reconstruction formulas can be formulated as a non-uniformly weighted (finite) ergodic average which reduces to the GLA version for unitary spectrum. It is also interesting to note that while the reconstruction formulas from section \ref{S=Snapshot-rec-theory} optimally reconstruct (by definition) any finite set of snapshots, they are not asymptotically consistent in the sense that they do not recover the correct expansion of the observable in terms of eigenvectors in the limit of infinite data. On the other hand, while the GLA reconstruction formulas are sub-optimal (for general operators possessing non-unitary spectrum) for any finite set of snapshots, they are asymptotically consistent, recovering the correct expansion in the limit. We also show that the least squares problem, solved via the g-reflexive inverses defined on a tensor product of coefficient spaces, is equivalent to an eigenvector-adapted hermitian form in state space.

\section{Preliminaries}\label{SS=Modal-appr-snapshots}
In this section, we set the scene and review results relevant for the later development. 
We assume that the data vectors $\set{\f_1,\f_2,\dots, \f_{m+1}} \subset \C^n$ are generated by a matrix $\A : \C^n \to \C^n$ via $\f_{i+1} = \A\f_i$ for some given initial $\f_1 \in \C^n$. As a technical simplification, a generic case is assumed, i.e. that $\A$ is diagonalizable, and that its Ritz values with respect to the subspaces spanned by the snapshots are simple. We can think of the snapshots $\f_i$ as being generated by e.g. black--boxed numerical software $\A$ for solving a partial differential equation with initial condition $\f_1$, or e.g. as vectorized images of flames in combustion chamber, taken by a high-speed camera.
%
%
The goal is to represent the snapshots $\f_i$ of the underlying dynamics using the (approximate) eigenvectors and eigenvalues of $\A$.

\subsection{Spectral analysis using Krylov basis}\label{SS=Modal-appr-snapshots}
The data driven framework leaves little room to maneuver; hence we resort to the classical Rayleigh-Ritz approximations from the Krylov subspaces spanned by the snapshots.  Let $\X_m$ be the column partitioned matrix $\X_m = ( \f_1 ~ \ldots ~ \f_m )$. The columns of $\X_m$ span the Krylov subpsace $\mc K_m = \linspan\set{\f_1,\f_2,\dots, \f_{m}}$.  Without loss of generality, we assume that $\X_m$ is of full rank. The action of $\A$ on $\mc{K}_m$ can be represented as

\begin{equation}\label{zd:eq:AK=KC+R}
\A \X_m = \X_m C_m + E_{m+1},\;\;  C_m = \begin{pmatrix} 0 & 0 & \ldots & 0 & c_1 \cr
1 & 0 & \ldots & 0 & c_2 \cr 
0 & 1 & \ldots & 0 & c_3 \cr
\vdots & \ddots & \ddots & \vdots & \vdots \cr
0 & 0 & \ldots & 1 & c_{m}\end{pmatrix},\;\;E_{m+1} = \mbf r_{m+1} \mbf e_m^T ,\;\; \mbf e_m=\begin{pmatrix} 0 \cr 0  \cr \vdots \cr 0 \cr 1\end{pmatrix} ,
\end{equation}
where the $c_i$'s are the $\X_m$-basis coefficients of the orthogonal projection of $\f_{m+1}$ onto $\mc K_m$, and, with $\mathbf{c}=(c_i)_{i=1}^m$, $\mbf r_{m+1}=\f_{m+1}-\X_m \mathbf{c}$ is orthogonal to $\mc K_{m}$;  $\X_m^* \mbf r_{m+1}=0$. Hence $C_m = (\X_m^* \X_m)^{-1}(\X_m^* \A \X_m)\equiv \X_m^{\dagger}\A\X_m$, where $\X_m^{\dagger}$ is the Moore--Penrose generalized inverse of $\X_m$. Thus, the Frobenius companion matrix $C_m$ is the matrix representation of the Rayleigh quotient $\mbb P_{\mc K_m} \restr{\A}{\mc K_m}$ in the basis formed by the columns of $\X_m$. In practical computation, the coefficients of $\mathbf{c}$ are obtained from the solution of the least squares problem $\| \X_m \mathbf{c} - \f_{m+1}\|_2\rightarrow\min$.

The spectral decomposition of $C_m$ has beautiful structure. Assume for simplicity that the eigenvalues $\lambda_i$, $i=1,\ldots, m$, are algebraically simple. It is easily checked that the rows of $\Vanderm_m$ are the left eigenvectors of $C_m$, so the columns of $\Vanderm_m^{-1}$ are the (right) eigenvectors of $C_m$; they are essentially unique.  Hence, the spectral decomposition of $C_m$ reads
\begin{equation}\label{zd:eq:C-V-Lambda}
C_m = \Vanderm_m^{-1}\Lambda_m \Vanderm_m,\;\;\mbox{where}\;\;
\Lambda_m = \begin{pmatrix} \lambda_1 &  & \cr
& \ddots &   \cr 
&        & \lambda_m\end{pmatrix},\;\;
\Vanderm_m = \begin{pmatrix} 
1 & \lambda_1 & \ldots & \lambda_1^{m-1} \cr
1 & \lambda_2 & \ldots & \lambda_2^{m-1} \cr
\vdots & \vdots & \ldots & \vdots \cr
1 & \lambda_m & \ldots & \lambda_m^{m-1} \cr
\end{pmatrix} .
\end{equation}

\begin{proposition} The columns of $\widehat{W}_m\equiv \X_m \Vanderm_m^{-1}$ are the Ritz vectors, with the corresponding Ritz values $\lambda_1, \ldots, \lambda_m$. For a Ritz pair $(\lambda_j, \widehat{W}_m(:,j))$, the residual is given by
\begin{equation}
\frac{\| \A \widehat{W}_m(:,j) - \lambda_j \widehat{W}_m(:,j)\|_2}{\|\widehat{W}_m(:,j)\|_2} = \frac{\|\mbf r_{m+1}\|_2}{\|\widehat{W}_m(:,j) \|_2} \prod_{\stackrel{k=1}{k\neq j}}^m \frac{1}{|\lambda_j-\lambda_k|} .
\end{equation}
\end{proposition}
\begin{proof}
It holds that $\A \widehat{W}_m = \widehat{W}_m \Lambda_m + \mbf r_{m+1}e_m^T\Vanderm_m^{-1}$,
where for the last row of $\Vanderm_m^{-1}$ we can use the formulas from \cite{Turner-Vadermonde-inverse-66} to obtain
$$
e_m^T\Vanderm_m^{-1} = \left(\begin{matrix} {\displaystyle \prod_{\stackrel{k=1}{k\neq 1}}^m \frac{1}{\lambda_1-\lambda_k}},\; {\displaystyle \prod_{\stackrel{k=1}{k\neq 2}}^m \frac{1}{\lambda_2-\lambda_k}},\;\ldots,\; {\displaystyle \prod_{\stackrel{k=1}{k\neq m-1}}^m \!\!\frac{1}{\lambda_{m-1}-\lambda_k}}, \; {\displaystyle\prod_{\stackrel{k=1}{k\neq m}}^m \frac{1}{\lambda_m-\lambda_k}} \end{matrix}\right) .
$$
 Hence, for a particular Ritz pair $(\lambda_j, \widehat{W}_m(:,j))$ we have
\begin{equation}
\A \widehat{W}_m(:,j) = \lambda_j \widehat{W}_m(:,j) + \mbf r_{m+1} \prod_{\stackrel{k=1}{k\neq j}}^m \frac{1}{\lambda_j-\lambda_k}.
\vspace{-9mm}
\end{equation}
\end{proof}

\subsubsection{Modal representation of the snapshots}\label{zd:SSS:modal-repr-V}
Once the Ritz pairs have provided useful spectral information on $\A$, we would like to analyze the snapshots $\f_i$ in terms of spectral data.  
%
To that end, write
\begin{equation}\label{zd:eq:X-W-V}
\X_m = (\X_m \Vanderm_m^{-1}) \Vanderm_m \equiv \widehat{W}_m \Vanderm_m, \;\;\A \X_m = \widehat{W}_m (\Lambda_m \Vanderm_m) + \mbf r_{m+1}e_m^T,
\end{equation}
and, with the column partition $\widehat{W}_m = \begin{pmatrix} \widehat{w}_1 & \ldots & \widehat{w}_m\end{pmatrix}$ of $\widehat{W}_m$, define 
\begin{equation}\label{zd:eq:Wm}
\mathfrak{a}_j = \|\widehat{w}_j\|_2,\;\; D_{\mathfrak{a}}=\mathrm{diag}(\mathfrak{a}_j)_{j=1}^m,\;\; W_m = \widehat{W}_m D_{\mathfrak{a}}^{-1}=\begin{pmatrix} {w}_1 & \ldots & {w}_m\end{pmatrix}. 
\end{equation}
Then, from the first relation in (\ref{zd:eq:X-W-V}), 
we have
\begin{equation}\label{zd:eq:f_i-dmd}
\f_i = \sum_{j=1}^m \widehat{w}_j \lambda_j^{i-1} \equiv \sum_{j=1}^m w_j \mathfrak{a}_j \lambda_j^{i-1},\;\;i=1,\ldots , m \; \Longleftrightarrow \X_m = \widehat{W}_m \Vanderm_m \equiv W_m D_{\mathfrak{a}} \Vanderm_m , 
\end{equation}
and from the last column in the second relation in (\ref{zd:eq:X-W-V}) we have
\begin{equation}\label{zd:eq:f_m+1-dmd}
\f_{m+1} = \sum_{j=1}^m \widehat{w}_j \lambda_j^m + \mbf r_{m+1} = \sum_{j=1}^m w_j \mathfrak{a}_j \lambda_j^m + \mbf r_{m+1} .
\end{equation}
It is important to note that the decompositions (or reconstructions) 
(\ref{zd:eq:f_i-dmd}, \ref{zd:eq:f_m+1-dmd}) are by definition attached to the Ritz pairs of $\A$, formed using the spectral decomposition (\ref{zd:eq:C-V-Lambda}) of the matrix representation $C_m$ of $\mathbb{P}_{\mathcal{K}_m} \restr{\A}{\mathcal{K}_m}$. If $\mathcal{K}_m$ is close to being an $\A$-invariant subspace, then $\A w_j \approx \lambda_j w_j$, i.e. the decompositions (\ref{zd:eq:f_i-dmd}, \ref{zd:eq:f_m+1-dmd}) are approximately in terms of the eigenpairs  of $\A$.

\begin{remark}\label{zd:REM:amplitudes}

		Note that $\|w_j\|_2=1$, $j=1,\ldots, m$, and that 
		$$
		D_{\mathfrak{a}}\Vanderm_m(:,1)=W_m^{\dagger}\X_m(:,1)=(\mathfrak{a}_j)_{j=1}^m\in\mathbb{R}^m, 
		$$		
		Of course, we can replace each term $w_j \mathfrak{a}_j$ with $(w_j\ee^{\ii\psi_j})(\ee^{-\ii\psi_j}\mathfrak{a}_j)$ (thus redefining $w_j$ and $\mathfrak{a}_j$) without affecting the decompositions (\ref{zd:eq:f_i-dmd}, \ref{zd:eq:f_m+1-dmd}), but this normalization to real $\mathfrak{a}_j$'s seems reasonable if we  interpret those numbers as amplitudes. Further,  the amplitudes and the scalings of the Ritz vectors can be done in some other appropriate norm instead of in $\|\cdot\|_2$. However if we pursued the computation of modes from more than one initial condition, the complex form of the amplitudes would be enforced by the requirement that modes are independent of initial conditions \cite{Mezic:2005}.
	
\end{remark}

\subsection{On computing the eigenvalues of $C_m$}
Since the spectral decomposition (\ref{zd:eq:C-V-Lambda}) of $C_m$ is given explicitly from its eigenvalues, it remains to compute the $\lambda_i$'s efficiently and in a numerically robust way. 
Computing the  eigenvalues of $C_m$ is equivalent to finding the zeros of its characteristic  polynomial $\wp_m(z) = z^m - \sum_{j=1}^m c_j z^{j-1}$, but this is only an elegant theoretical connection that has limited value in practical computation. In fact, the most robust polynomial root finding procedure is based on solving the matrix eigenvalue problem, while taking the structure of the companion matrix $C_m$ into account. For an excellent mathematical elucidation we refer to  \cite{edelman-murakami-roots-1995}, \cite{PAN20111305}. 

From the software point of view, if we want truly high performance, both in terms of numerical robustness and run time efficiency, using an eigenvalue method for general matrices (such as e.g. \texttt{eig()} in Matlab) is not the best choice. Namely, in that case the eigenvalues are extracted from the Schur form with the backward error $\delta C_m$ that is small in the sense that $\|\delta C_m\|_2/ \|C_m\|_2$ is of the order of the machine roundoff $\roff$, but $C_m+\delta C_m$ is not a companion matrix and there is no information on the size of the backward error in the coefficients $c_i$. Further, the computation requires $O(m^2)$ memory space and $O(m^3)$ \emph{flops}.

Proper method for this case is the one that preserves the structure of the companion matrix and for which the computed eigenvalues correspond exactly to the eigenvalues of a companion matrix $\widetilde{C}_m$ with coefficients $\widetilde{c}_i=c_i + \delta c_i$, where $\max_i |\delta c_i| / |c_i|$ is small. Note that backward stability in terms of the coefficients $c_i$ is proper framework for assessing the numerical accuracy as the coefficients $c_i$ are the results from previous computation -- solving the least squares problem $\| \X_m \mathbf{c} - \f_{m+1}\|_2\rightarrow\min$. The desired complexity is $O(m)$ in memory space and $O(m^2)$ in \emph{flop} count. 

Two methods that satisfy the above requirements are presented in 
\cite{BINI20102006}, \cite{Aurentz-Mach-Raf-David-2015} and they should be used in high performance software implementations.
%
%
Particularly interesting is the unitary--plus--rank one formulation 
\begin{equation}\label{zd:eq:C=U+r1}
C_m \! =\! \begin{pmatrix} 0 & 0 & \ldots & 0 & c_1 \cr
1 & 0 & \ldots & 0 & c_2 \cr 
0 & 1 & \ldots & 0 & c_3 \cr
\vdots & \ddots & \ddots & \vdots & \vdots \cr
0 & 0 & \ldots & 1 & c_m\end{pmatrix} \! = \! 
\begin{pmatrix} 0 & 0 & \ldots & 0 & \pm 1 \cr
1 & 0 & \ldots & 0 & 0 \cr 
0 & 1 & \ldots & 0 & 0 \cr
\vdots & \ddots & \ddots & \vdots & \vdots \cr
0 & 0 & \ldots & 1 & 0\end{pmatrix} + \begin{pmatrix} \mp 1 + c_1 \cr c_2 \cr c_3 \cr \vdots \cr c_m \end{pmatrix} \! \begin{pmatrix} 0 & 0 & \ldots & 0 & 1 \end{pmatrix} \equiv U_m + \widehat{\mathbf{c}} \mbf e_m^T
\end{equation}
which has been exploited in \cite{BINI20102006} to construct an efficient implicit  QR iterations process that requires $O(m)$ memory and $O(m^2)$ \emph{flops}. This splitting, when plugged into (\ref{zd:eq:AK=KC+R}), provides the following intersting form of the Krylov decomposition.
\begin{proposition} If in the splitting (\ref{zd:eq:C=U+r1}) we choose $(U_m)_{1,m}=1$, then we can write (\ref{zd:eq:AK=KC+R}) as 
\begin{equation}
\A \X_m = \X_m U_m + (\f_{m+1}-\f_1)e_m^T.
\end{equation}
\end{proposition}
\begin{proof}
Note that $\A\X_m = \X_m U_m + ( ( \X_m \mathbf{c}- \f_1)   + \mbf r_{m+1})\mbf e_m^T$, where $\mbf r_{m+1} = \f_{m+1}-\X_m \mathbf{c}$.	
\end{proof}
%
%
%
%

An important remark is in order.
\begin{remark}\label{zd:REM:eig-vec-sens}
If the eigenvalues $\lambda_i$'s of $C_m$ are computed as $\widetilde{\lambda}_i$'s with small backward error in the vector of the coefficients $\mathbf{c}$, then
$\Vanderm_m(\widetilde{\lambda}_i)^{-1}$ is (assuming that all $\widetilde{\lambda}_i$'s are algebraically simple) the exact eigenvector matrix of the companion matrix $\widetilde{C}_m$ defined by $\mathbf{c}+\delta\mathbf{c}$, where $\|\delta\mathbf{c}\|_2/\|\mathbf{c}\|_2$ is small.	
To appreciate this fact more, consider the general case. If we compute the eigenvectors of a general square matrix $S$, then the accuracy depends on the condition number of the eigenvectors and on the gap between an eigenvalue and it neighbors in the spectrum. More precisely, if $\lambda$ is a simple eigenvalue of a diagonalizable $S$, with unit eigenvector $y$, then with appropriate nonsingular $Y=( y\; Y_2)$ and $Z=(z \; Z_2)=Y^{-1}$, $Z^* S Y = \left(\begin{smallmatrix} \lambda & \0 \cr \0 & \Lambda_2\end{smallmatrix}\right)$ is diagonal. If $(\widetilde{\lambda},\widetilde{y})$ is an eigenpair of $S+E$, then, under some additional assumptions,
$$
\| \widetilde{y}-y\|_2 \leq \| (\lambda I - \Lambda_2)^{-1}\|_2 \|Y_2\|_2 \|Z_2\|_2 \|E\|_2 \leq \frac{\kappa_2(Y)}{\min_{j}|\lambda - (\Lambda_2)_{jj}|} \|E\|_2.
$$
 So, for instance, if a group of eigenvalues is tightly clustered, then their corresponding eigenvectors are extremely sensitive and difficult to compute numerically. Here we can set $E=-r\widetilde{y}^*$, where $r=S\widetilde{y}-\widetilde{\lambda}\widetilde{y}$ is the residual. For more details ee e.g. \cite{eis-ipsen-bit-1998}, \cite[Ch. V.,\S 2]{ste-sun-90}, \cite[\S 3.2.2]{Bjorck-book-2015}.
\end{remark}

\subsection{Computation with Vandermonde matrices}\label{SS=Comp-Vandermonde}
The natural representation of the evolving dynamics $\f_{i+1}=\A \f_i$ by the Krylov decomposition (\ref{zd:eq:AK=KC+R}) and the simple and elegant snapshots' decompositions (\ref{zd:eq:f_i-dmd}, \ref{zd:eq:f_m+1-dmd}) have not lead to a numerical scheme that can be used in practical computations.
The reason is in numerical difficulties when computing the matrix $\widehat{W}_m = \X_m \Vanderm_m^{-1}$ of the Ritz vectors, due to potentially extremely high condition number of the Vandermonde matrix $\Vanderm_m$.

And indeed, the Vandermonde matrices can be arbitrarily badly ill-conditioned \cite{Pan-How-bad}. More precisely, the condition number $\kappa_2(\Vanderm_m)\equiv \|\Vanderm_m\|_2 \|\Vanderm_m^{-1}\|_2$ will depend on the distribution of the eigenvalues $\lambda_i$ and it can be as small as one  and it can grow with the dimension $m$ as fast as $O(m^{m+1})$ for harmonically distributed $\lambda_i$'s, see Gautschi \cite{Gautschi1975}, \cite{Gautschi-How-Unstable}. For example, any real $n\times n$ Vadermonde matrix has condition number greater than $2^{n-2}/\sqrt{n}$.
As an illustration, we show on Figure \ref{zd:fig:V20_condition} the values of 
$\kappa_2(\Vanderm_{20})$ on three sets, each containing $100$ matrices. Recall that the classical upper bound on the relative error in the solution of linear systems is $O(n)\roff\kappa_2(\Vanderm_m)$, so that $\kappa_2(\Vanderm_m) > 1/(n\roff)$ implies no accuracy whatsoever.

On the other hand, if the $\lambda_i$'s are the $m$th roots of unity, then $\Vanderm_m$ is the unitary Discrete Fourier Tranformation (DFT) matrix with $\kappa_2(\Vanderm_m)=1$. In fact, if the $\lambda_i$'s are on the unit circle, then $\kappa_2(\Vanderm_m)$ is, under some additional assumptions (e.g. that the nodes are not tightly clustered), usually moderate \cite{fermin-bazan-vanderm-2000}, \cite{berman-feuer-vanderm-2007}.
If the underlying operator is nearly unitary then the $\lambda_i$'s will be close to the unit circle and $\Vanderm_m$ is usually well conditioned. Further, it is well-known that the condition number can be reduced by scaling and this opens a possibility for accurate reconstruction, at least in those
cases when scaling sufficiently reduces the condition number. 

 With an unlucky distribution of the nodes, a Vandermonde matrix can be extremely ill-conditioned even for small dimensions, e.g. if $|\xi|\ll 1$ is small, then both matrices
 $$
 \Vanderm_2 = \left( \begin{smallmatrix} 1 & 1-\xi \cr 1 & 1+\xi \end{smallmatrix}\right),\;\;
 \Vanderm_3 = \left( \begin{smallmatrix} 1 & -1 & 1  \cr 1 & 1 & 1 \cr
 1 & 1+\xi & (1+\xi)^2 \end{smallmatrix}\right)
 $$
 can be turned into singular ones by small $O(|\xi|)$ changes, so their condition numbers are $O(1/|\xi|)\gg 1$. Furthermore, the rows as well as the columns of both $\Vanderm_2$ and $\Vanderm_3$ above are nearly equilibrated (nearly of same $\ell_1$, $\ell_2$ or $\ell_\infty$ norm) so that the condition number cannot be improved by diagonal scalings.

\begin{figure}[H]
	\begin{center}
		\includegraphics[width=\linewidth, height=2.5in]{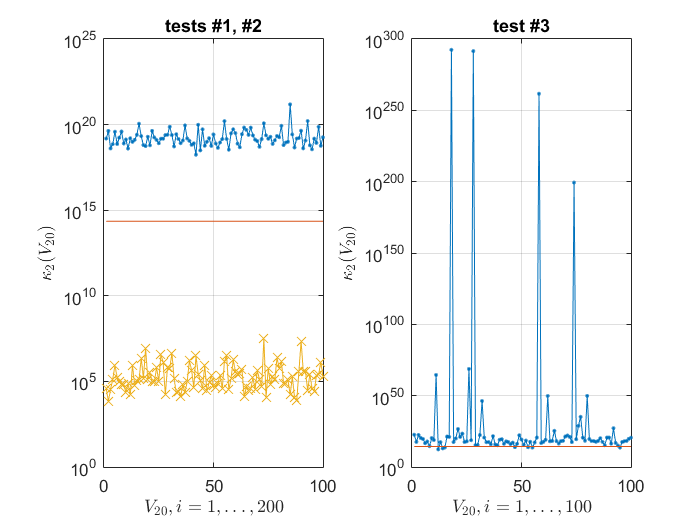}
\includegraphics[width=0.45\linewidth, height=2.5in]{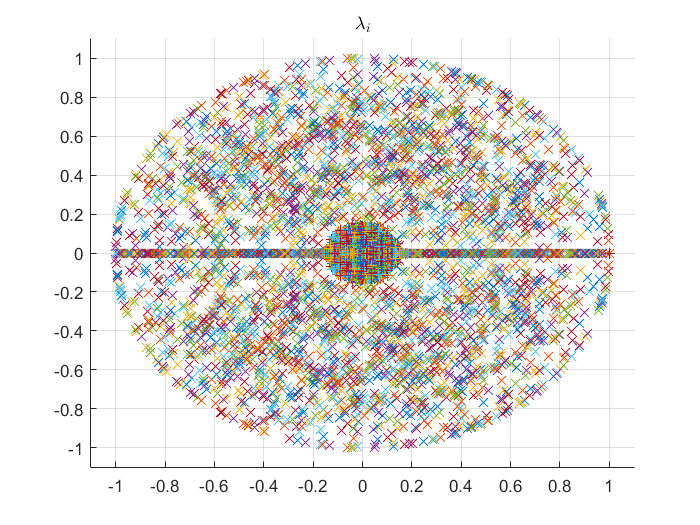}
		\includegraphics[width=0.45\linewidth, height=2.5in]{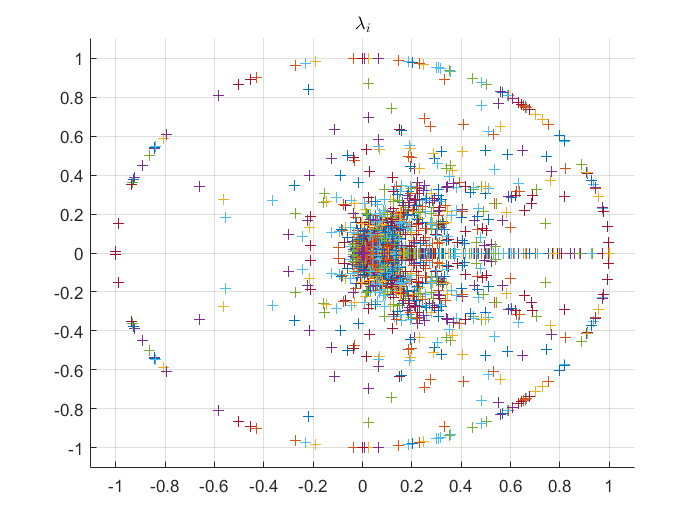}
	\end{center}	
	\caption{\label{zd:fig:V20_condition} The spectral condition number over three sets of the total of $300$ Vandermonde matrices of dimension $m=20$, $\Vanderm_{20}(\lambda_i)$. The $\lambda_i$'s are the eigenvalues of pseudo-random  matrices $A$, scaled to have unit spectral radius. \emph{Left panel}: First, $100$ matrices are generated in Matlab as $A=\texttt{rand(20,20)}$, $A=A/\texttt{max(abs(eig(A)))}$; the condition numbers ($\cdot-$) of the thus generated Vandermonde matrices are all above $10^{15}$. Then, $100$ matrices are generated as  $A=\texttt{randn(20,20)}$, $A=A/\texttt{max(abs(eig(A)))}$; in all cases $\kappa_2(\Vanderm_{20}(\lambda_i)) < 10^{10}$. \emph{Right panel}: $100$ samples of $\Vanderm(\lambda_i)$ are generated using the eigenvalues of $A=\texttt{expm(-inv(rand(n,n)))}$, $A=A/\texttt{max(abs(eig(A)))}$. The horizontal line is the value of $1/(n\roff)\approx \texttt{2.25e+14}$. The plots in the second row show all $\lambda_i$'s used to generate the matrices $\Vanderm_{20}(\lambda_i)$}
\end{figure}


In certain cases, accurate computation is possible independent of the condition number. The Bj\"{o}rck-Pereyra algorithm \cite{BjoP70} can solve the Vandermonde system to high accuracy
if the $\lambda_i$'s are real, of same sign and increasingly ordered, see \cite{Hig87b}. It is the distribution of the signs that determines whether a catastrophic cancellation can occur and in our case, unfortunately, we cannot rely on such a constellation of the Ritz values as the $\lambda_i$'s are in general complex numbers.
Our main goal in this work is to provide a numerically robust implementation of the snapshot reconstruction scheme outlined in \S \ref{SS=Modal-appr-snapshots},  that is independent of the distribution of the Ritz values $\lambda_i$. 

%% file: SOURCES/Preliminaries.tex
\subsection{Reconstruction using Schmid's DMD -- an alternative to working with Vandermonde matrices}\label{SS=Reconstruct_DMD} 
To alleviate numerical difficulties caused by the inherent ill-conditioning of the matrix $\X_m$ and of the eigenvalues of the companion matrix, 
Schmid \cite{schmid2010} proposed using, instead of the companion matrix $C_m$, the Rayleigh quotient $S_m = U_m^* \A U_m$, where $\X_m=U_m \Sigma_m V_m^*$ is the SVD of $\X_m$. More precisely, the SVD is used to determine a numerical rank $k$ of $\X_m$, and then the Rayleigh quotient is $S_k=U_k^* \A U_k$, where the columns of $U_k$ are the leading $k$ left singular vectors of $\X_m$. The resulting scheme, designated as DMD (Dynamic Mode Decomposition) and outlined in Algorithm \ref{zd:ALG:DMD} below, has become a tool of trade in computational fluid dynamics.


\begin{algorithm}[H]
	\caption{{$[Z_k, \Lambda_k]=\mathrm{DMD}(\X_m,\Y_m)$}}
	\label{zd:ALG:DMD}
	\begin{algorithmic}[1]
		\REQUIRE \  \\		
		\begin{itemize} 
			\item $\X_m=(\x_1,\ldots,\x_m), \Y_m=(\y_1,\ldots,\y_m)\in \mathbb{C}^{n\times m}$ that define a sequence of snapshots pairs $(\x_i,\y_i\equiv \A \x_i)$. (Tacit assumption is that $n$ is large and that $m \ll n$.)
		\end{itemize}
		\STATE $[U,\Sigma, \Phi]=svd(\X_m)$ ; \COMMENT{\emph{The thin SVD: $\X_m = U \Sigma \Phi^*$, $U\in\mathbb{C}^{n\times m}$, $\Sigma=\mathrm{diag}(\sigma_i)_{i=1}^m$}}
		\STATE Determine numerical rank $k$.
		\STATE Set $U_k=U(:,1:k)$, $\Phi_k=\Phi(:,1:k)$, $\Sigma_k=\Sigma(1:k,1:k)$ 	
		\STATE ${S}_k = (({U}_k^* \Y_m) \Phi_k)\Sigma_k^{-1}$; \COMMENT{\emph{Schmid's formula for the Rayleigh quotient $U_k^* \A U_k$}}
		\STATE $[B_k, \Lambda_k] = \mathrm{eig}(S_k)$ \COMMENT{$\Lambda_k=\mathrm{diag}(\lambda_j)_{j=1}^k$; $S_k B_k(:,j)=\lambda_j B_k(:,j)$; $\|B_k(:,j)\|_2=1$}\\
		\COMMENT{\emph{We always assume the generic case that $S_k$ is diagonalizable, i.e. that in Line 5. the function \texttt{eig()} computes the full column rank matrix $B_k$ of the eigenvectors of $S_k$.}}
		\STATE $Z_k = U_k B_k$ \COMMENT{\emph{Ritz vectors}}
		\ENSURE $Z_k$, $\Lambda_k$
	\end{algorithmic}
\end{algorithm}
\noindent In this way, the matrix $\widehat{W}_m$ of the Ritz vectors and the amplitudes for the representation of the snapshot are computed without having to invert Vandermonde matrix. The two approaches are algebraically equivalent, and in the following proposition we provide the details.

\begin{proposition}\label{PROP:amplitudes} The amplitudes $\mathfrak{a}_j$ in the decompositions (\ref{zd:eq:f_i-dmd}, \ref{zd:eq:f_m+1-dmd}) can be equivalently computed as
	\begin{equation}\label{eq:a_j-formulas}
	\mathfrak{a}_j \equiv \| \X_m\Vanderm_m^{-1}e_j\|_2 = |Z_m^\dagger\X_m(:,1)|,\;\;j=1,\ldots, m.
	\end{equation}
\end{proposition}
\begin{proof} Let $S_m b_j = \lambda_j b_j$, with $b_j=B_m(:,j)$, $\|b_j\|_2=1$. Since
$$
C_m = (\X_m^* \X_m)^{-1}(\X_m^* \A \X_m) = (\Phi_m\Sigma_m^{-2}\Phi_m^*)(\Phi_m\Sigma_m U_m^*) \A (U_m\Sigma_m \Phi_m^*) = \Phi_m \Sigma_m^{-1}(U_m^*\A U_m)\Sigma_m\Phi_m^*,
$$
we have $S_m = \Sigma_m \Phi_m^* C_m \Phi_m \Sigma_m^{-1}$, and  
$
C_m (\Phi_m \Sigma_m^{-1} b_j) = \lambda_j (\Phi_m\Sigma_m^{-1}b_j) .
$
Further, since $\lambda_j$ is assumed simple, $\Phi_m\Sigma_m^{-1}b_j$ and $(\Vanderm_m^{-1})(:,j)$ are collinear and thus, $b_j$ is collinear with $\Sigma_m \Phi_m^* (\Vanderm_m^{-1})(:,j)$. Note that 
$$
\|\Sigma_m \Phi_m^* (\Vanderm_m^{-1})(:,j)\|_2 = \|U_m \Sigma_m \Phi_m^* (\Vanderm_m^{-1})(:,j) \|_2 = \| \X_m (\Vanderm_m^{-1})(:,j)\|_2 .
$$ 
Since $\|b_j\|_2=1$, it holds with some $\psi_j\in\mathbb{R}$,
$$
b_j = \Sigma_m \Phi_m^* (\Vanderm_m^{-1})(:,j) \frac{\ee^{\ii\psi_j}}{\mathfrak{a}_j} ,
$$ 
and the corresponding Ritz vector $z_j = Z_m(:,j)$ then reads
$$
 z_j = U_m b_j = U_m \Sigma_m \Phi_m^* (\Vanderm_m^{-1})(:,j) \frac{\ee^{\ii\psi_j}}{\mathfrak{a}_j} = \X_m (\Vanderm_m^{-1})(:,j) \frac{\ee^{\ii\psi_j}}{\mathfrak{a}_j} = \widehat{W}_m(:,j) \frac{\ee^{\ii\psi_j}}{\mathfrak{a}_j} = W_m(:,j) \ee^{\ii\psi_j} .
$$
Hence, $Z_m = W_m \Psi$, $\Psi=\mathrm{diag}(\ee^{\ii\psi_j})_{j=1}^m$, and we seek reconstruction in the form 
\begin{equation}\label{eq:reconstruction-formula}
\X_m = \begin{pmatrix} z_1 & z_2 & \ldots & z_m \end{pmatrix} \begin{pmatrix} 
\widetilde{\mathfrak{a}}_1 &  &  & \cr
               & \widetilde{\mathfrak{a}}_2 &  &   \cr 
  &  & \ddots &        \cr 
&        &  & \widetilde{\mathfrak{a}}_m\end{pmatrix}
\begin{pmatrix} 
1 & \lambda_1 & \ldots & \lambda_1^{m-1} \cr
1 & \lambda_2 & \ldots & \lambda_2^{m-1} \cr
\vdots & \vdots & \ldots & \vdots \cr
1 & \lambda_m & \ldots & \lambda_m^{m-1} \cr
\end{pmatrix} .
\end{equation}
Note that such a decomposition exists by (\ref{zd:eq:f_i-dmd}), and (\ref{eq:reconstruction-formula}) is just a way of expressing it in terms of Algorithm \ref{zd:ALG:DMD}. 
If we equate the first columns on both sides in the above relation, then solving for the $\widetilde{\mathfrak{a}}_j$'s and using (\ref{zd:eq:f_i-dmd}) yields\footnote{Note that $Z_m^\dagger Z_m=I_m$.}
\begin{equation}\label{zd:eq:DMD-aplitudes}
(\widetilde{\mathfrak{a}}_j)_{j=1}^m = Z_m^\dagger\X_m(:,1)=\Psi^* W_m^{\dagger}\X_m(:,1) = \Psi^* \cdot (\mathfrak{a}_j)_{j=1}^m = (\ee^{-\ii\psi_j}\mathfrak{a}_j)_{j=1}^m .
\end{equation}
In terms of the quantities computed by Schmid's method, the representations (\ref{zd:eq:f_i-dmd}, \ref{zd:eq:f_m+1-dmd}) hold with $z_j\widetilde{\mathfrak{a}}_j$ instead of $w_j \mathfrak{a}_j$.
As discussed in Remark \ref{zd:REM:amplitudes}, we can scale $z_j$ by $\ee^{-\ii\psi_j}={\widetilde{\mathfrak{a}}_j^*}/|\widetilde{\mathfrak{a}}_j|$ and replace $\widetilde{\mathfrak{a}}_j$ with 
$\mathfrak{a}_j = |\widetilde{\mathfrak{a}}_j|$. 
\end{proof}

If Algorithm \ref{zd:ALG:DMD} uses $k<m$, then $S_k$ is a Rayleigh quotient of $C_m$, and the columns of $Z_k$ are not the same Ritz vectors as in $W_m$ from (\ref{zd:eq:Wm}). For more details on this connection, we refer to \cite{DDMD-SISC-2018}.

\begin{remark}
In the framework of the Schmid's DMD, the amplitudes are usually determined by the formula (\ref{zd:eq:DMD-aplitudes}). Since $Z_m = U_m B_m$ (see line 6. in Algorithm \ref{zd:ALG:DMD}) and $U_m^* U_m=I_m$, instead of applying the pseudoinverse of the explicitly computed $Z_m$, we use more the more efficient formula
\begin{equation}
Z_m^\dagger \X_m(:,1) = B_m^{-1}(U_m^* \X_m(:,1))=B_m^{-1} U_m^* U_m\Sigma_m \Phi_m^* e_1 = B_m^{-1} (\Sigma_m \Phi_m(1,:)^*).
\end{equation}
Recall that we assume that all $\lambda_j$'s are mutually distinct, so $B_m$ is of full rank. Since it can be ill-conditioned, we can use the (possibly truncated) SVD of $B_m$ and determine the $\widetilde{\mathfrak{a}}_j$'s as least squares solution. In the case of numerical rank deficiency in $B_m$, we can choose/prefer sparse solution (instead of the least norm).
\end{remark}

\begin{remark}
Recently, in \cite{DDMD-SISC-2018}, we proposed  a new computational scheme -- Refined Rayleigh Ritz Data Driven Modal Decomposition. It follows the DMD scheme, but it further allows data driven refinement of the Ritz vectors and computable data driven residuals. 

\end{remark}

%% file: SOURCES/Section_Vanderm_DFT.tex
\section{Numerical algorithm for computing $\widehat{W}_m=\X_m\Vanderm_m^{-1}$}\label{zd:SS=DFT+Cauchy-solver}
There is certain intrinsic elegance in the snapshots reconstructions (\ref{zd:eq:f_i-dmd}, \ref{zd:eq:f_m+1-dmd}) based on the spectral structure of the companion matrix (\ref{zd:eq:C-V-Lambda}). Unfortunately, the potentially high condition number of $\Vanderm_m$ precludes exploiting it in numerical computations. Indeed, even in relatively simple examples we can encounter $\kappa_2(\Vanderm_m)$ as high as $10^{50}$, $10^{70}$ or higher. If $\widehat{W}_m=\X_m\Vanderm_m^{-1}$ is computed in the standard double precision arithmetic with the roundoff $\roff\approx 2.2\cdot 10^{-16}$, the classical perturbation theory estimates the relative error in the computed result essentially as\footnote{Up to an factor that is polynomial in the dimensions of the problem.} $\roff \kappa_2(\Vanderm_m)$ (e.g. $2.2\cdot 10^{-16}\cdot 10^{50}$), thus rendering the output as entirely wrong and useless. The Shmid's DMD provides an alternative path that avoids $\Vanderm_m^{-1}$, as outlined in \S \ref{SS=Reconstruct_DMD}.

We now explore another way to curb ill-conditioning, based on the methods of numerical linear algebra \cite{dgesvd-99}, \cite{Demmel-99-AccurateSVD}, \cite{demmel-koev-2006-PolynVandermonde}. Our starting point, based on the state of the art in numerical linear algebra, is that accurate computation with notoriously ill-conditioned Vandermonde matrices may be possible despite high classical condition number.

\subsection{A review of accurate computation with Vandermonde matrices}
For the reader's convenience, we briefly describe the two main components of computational schemes capable of performing linear algebra operations with $\Vanderm_m$ accurately in standard floating point arithmetic. 

First, the Discrete Fourier Transform (DFT) of a Vandermonde matrix $\Vanderm_m$ is a generalized Cauchy matrix parametrized by the $\lambda_i$'s (the original parameters that define $\Vanderm_m$) and the $m$th roots of unity. In our setting, DFT is an allowable and actually meaningfully interpretable transformation, and the result obtained by using the auxiliary Cauchy matrix are easily back substituted into the original formulation. The details are given in \S \ref{SS=DFT(V)}. 

Secondly,  in the framework of \cite{dgesvd-99}, \cite{Demmel-99-AccurateSVD},  accurate computation with generalized Cauchy matrix is possible (e.g. LDU decomposition, SVD) independent of its condition number.

\subsubsection{Discrete Fourier Transform of $\Vanderm_m$}\label{SS=DFT(V)}
Let $\DFT$ denote the Discrete Fourier Transform (DFT) matrix, $\DFT_{ij}=\omega^{(i-1)(j-1)}/\sqrt{m}$, where $\omega=\ee^{2\pi \ii/m}$, $\ii=\sqrt{-1}$.
Now, recall that DFT transforms Vandermonde into Cauchy matrices as follows (see \cite{Demmel-99-AccurateSVD}):
\begin{equation}\label{zd:eq:V*DFT}
(\Vanderm_m \DFT)_{ij} = \left[\frac{\lambda_i^m-1}{\sqrt{m}}\right] \left[\frac{1}{\lambda_i -\omega^{1-j}}\right] \left[{\omega^{1-j}}\right]
\equiv (\mathcal{D}_1)_{ii}\; \Cauchy_{ij}\; (\mathcal{D}_2)_{jj},\;1\leq i, j\leq m.
\end{equation}
In other words, $\Vanderm_m \DFT = \mathcal{D}_1 \Cauchy \mathcal{D}_2$ where $\mathcal{D}_1$ and $\mathcal{D}_2$ are diagonal, and $\Cauchy$ is a Cauchy matrix. Note, in addition, that $\mathcal{D}_2=\mathrm{diag}(\omega^{1-j})_{j=1}^m$ is unitary.

To avoid singularity (i.e. an expression of the form $0/0$) if in (\ref{zd:eq:V*DFT}) some $\lambda_i$'s equals an $m$-th root of unity we proceed as follows. If $\lambda_i=\omega^{1-j}$ for some index $j$, 
write $\lambda_i^m - 1 = \prod_{k=1}^m (\lambda_i - \omega^{1-k})$ and replace 
(\ref{zd:eq:V*DFT}) with the equivalent formula for the $i$-th row
\begin{equation}\label{zd:eq:V*DTF-general}
(\Vanderm_m \DFT)_{ij} = \underbrace{\frac{1}{\sqrt{m}}}_{(\mathcal{D}_1)_{ii}} \prod_{\stackrel{k=1}{k\neq j}}^m (\lambda_i - \omega^{1-k}) \underbrace{\omega^{1-j}}_{(\mathcal{D}_2)_{jj}},\;\; (\Vanderm_m \DFT)_{ik}=0\;\;\mbox{for}\;k\neq j.
\end{equation}
If $\varsigma=(\varsigma_1,\ldots,\varsigma_\ell)$ is a subsequence of $(1,\ldots,n)$, and if $\Vanderm_{\varsigma,m}$ is an $\ell\times m$ submatrix of $\Vanderm_m$ consisting of the rows with indices $\varsigma_1, \ldots, \varsigma_\ell$, then (\ref{zd:eq:V*DFT}, \ref{zd:eq:V*DTF-general}) trivially hold for the entries of $\Vanderm_{\varsigma,m} \DFT$.

Note that the matrices $\mathcal{D}_1$, $\Cauchy$,  $\mathcal{D}_2$ are given implicitly by the parameters $\lambda_i$ (eigenvalues, available on input) and the $m$-th roots of unity $\zeta_j = \omega^{1-j}$, $j=1,\ldots, m$ (easily computed to any desired precision and e.g. tabulated in a preprocessing phase of the computation), so that the DFT $\Vanderm_m\DFT$ is not done by actually running an FFT. It suffices to make a  note that the $\lambda_i$'s and the roots of unity are the parameters that define $\Vanderm_m\DFT$ as in (\ref{zd:eq:V*DFT}, \ref{zd:eq:V*DTF-general}).

\begin{remark}\label{zd:REM:Krylov-DFT}
It is interesting to see how this transformation $\Vanderm_m \mapsto \Vanderm_m\DFT$ fits the framework of the Krylov decomposition (\ref{zd:eq:AK=KC+R}). 
Post-multiply (\ref{zd:eq:AK=KC+R}) with $\DFT$ to obtain
\begin{equation}
\A (\X_m \DFT) = (\X_m \DFT)\DFT^* (\Vanderm_m^{-1}\Lambda_m \Vanderm_m)\DFT + r_{m+1}e_m^T \DFT
\end{equation}
and then, using (\ref{zd:eq:V*DFT}), $\DFT^* C_m \DFT = \DFT^*(\Vanderm_m^{-1}\Lambda_m \Vanderm_m)\DFT = \mathcal{D}_2^*\Cauchy^{-1} \mathcal{D}_1^{-1}\Lambda_m \mathcal{D}_1 \Cauchy \mathcal{D}_2 = \mathcal{D}_2^*\Cauchy^{-1}\Lambda_m \Cauchy \mathcal{D}_2$ and 
\begin{eqnarray}
\A (\X_m \DFT) &=& (\X_m\DFT) ((\Cauchy\mathcal{D}_2)^{-1}\Lambda_m (\Cauchy\mathcal{D}_2)) + r_{m+1}e_m^T \DFT,\;\;\mbox{or, equivalently,}  \label{eq:Krylov-DFT-1}\\
\A (\X_m \DFT \mathcal{D}_2^*) &=& (\X_m\DFT \mathcal{D}_2^*) (\Cauchy^{-1}\Lambda_m \Cauchy) + r_{m+1}e_m^T \DFT \mathcal{D}_2^* .\label{eq:Krylov-DFT-2}
\end{eqnarray}
If we think of each row $\X_m(i,:)$ as a time trajectory of the corresponding observable, then $\X_m(i,:)\DFT$ represents its image in the frequency domain, and (\ref{eq:Krylov-DFT-1}, \ref{eq:Krylov-DFT-2}) is the corresponding Krylov decomposition. 

\end{remark}

\subsubsection{Rank revealing (LDU) decomposition of $\mathcal{D}_1 \Cauchy \mathcal{D}_2$}\label{SSS=LDU(Cauchy)}
Applying the DFT to $\Vanderm_m$ in order to avoid the ill--conditioning of $\Vanderm_m$ may seem a futile effort -- since $\DFT$ is unitary, $\kappa_2(\mathcal{D}_1 \Cauchy \mathcal{D}_2)=\kappa_2(\Vanderm_m\DFT)=\kappa_2(\Vanderm_m)$.
Further, Cauchy matrices are also notoriously ill-conditioned, so, in  essence, we have traded one badly conditioned structure to another one. 

\begin{example}\label{EX:Hilb100}
The best known example of ill--conditioned Cauchy matrix  is the Hilbert matrix, $H_{ij}=1/(i+j-1)$.  For instance, the condition number of the $100\times 100$ Hilbert matrix satisfies $\kappa_2(H)>10^{150}$.
Ill-conditioning is not always obvious in the sizes of its entries -- the entries of the $100\times 100$ Hilbert matrix range from $1/199\approx 5.025\cdot 10^{-3}$ to $1$. Moreover, in Matlab, \texttt{cond(hilb(100))} returns \texttt{ans=4.622567959141155e+19}. One should keep in mind that the matrix condition number is a matrix function with its own condition number. By a result of Higham \cite{HIGHAM1995193}, condition number of the condition number is the condition number itself, meaning that our computed condition number, if it is above $1/\roff$ (in Matlab, \texttt{1/eps=4.503599627370496e+15}), it might be entirely wrongly computed. This may lead to an underestimate of extra precision needed to handle the ill--conditioning.	

\end{example}
Although we have not changed the condition number, we have changed the representation of the data, which will allow more accurate computation.
The key numerical advantage of this change of variables is in the fact that for any two diagonal matrices $\mathcal{D}_1$, $\mathcal{D}_2$ and any Cauchy matrix $\Cauchy$, the pivoted LDU decomposition 
\begin{equation}\label{zd:eq:Cauchy-LDU}
\Pi_1 (\mathcal{D}_1 \Cauchy \mathcal{D}_2)\Pi_2 = L \Delta U
\end{equation}
can be computed by a specially tailored algorithm so that all entries of $L$, $\Delta$, $U$, even the tiniest ones, are computed to nearly machine precision accuracy, no matter how high is the condition number of $\mathcal{D}_1 \Cauchy \mathcal{D}_2$. More precisely, if $\widetilde{L}$, $\widetilde{\Delta}$ and $\widetilde{U}$ are the computed matrices, then, for all $i, j$, 
\begin{equation}\label{zd:eq:Cauchy-LDU-error-bound}
|\widetilde{L}_{ij} - L_{ij}| \leq \epsilon |L_{ij}|,\;\;
|\widetilde{\Delta}_{ii} - \Delta_{ii}| \leq \epsilon |\Delta_{ii}|,\;\;
|\widetilde{U}_{ij} - U_{ij}| \leq \epsilon |U_{ij}|,\;\;
\end{equation} 
where $L\Delta U$ is the pivoted LDU decomposition that is computed exactly from the stored parameters.
Essentially, if we take the stored (in the machine memory) eigenvalues $\lambda_i$ and the roots of unity $\omega^{1-j}$ as our initial data, the first errors committed in the floating point  LDU decomposition (\ref{zd:eq:Cauchy-LDU}) are the entry-wise small forward errors (\ref{zd:eq:Cauchy-LDU-error-bound}).
This is achieved by avoiding subtractions of intermediate results and using clever updates of the Schur complements, see \cite{Demmel-99-AccurateSVD}. 

Moreover, as a consequence of pivoting, the matrix $L$ (lower triangular with unit diagonal) and the matrix $U$ (upper triangular with unit diagonal) are well conditioned. All ill-conditioning is conspicuously exposed on the diagonal of $\Delta$. For instance, in the case of the Hilbert matrix from Example \ref{EX:Hilb100}, $\kappa_2(L)=\kappa_2(U)\approx 72.34$ and $\kappa_2(\Delta)\approx 10^{149}$.

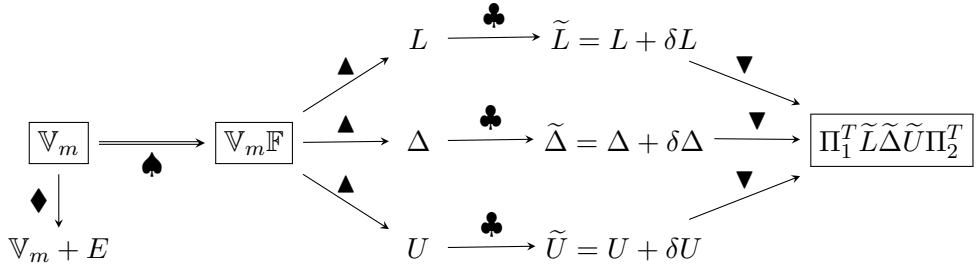
\begin{figure}[H]
	\begin{center}
\begin{tikzpicture}
\matrix (m) [matrix of math nodes,row sep=1.5em,column sep=3em,minimum width=2em]
{
            &                  & L & \widetilde{L}=L+\delta L & \\
\boxed{\Vanderm_m}  & \boxed{\Vanderm_m \DFT}  & \Delta & \widetilde{\Delta} =\Delta+\delta \Delta & \boxed{\Pi_1^T \widetilde{L}\widetilde{\Delta}\widetilde{U}\Pi_2^T} \\ \Vanderm_m + E            &                  & U & \widetilde{U}=U+\delta U & \\};
\path[-stealth]
(m-2-1) edge [double] node [below] {$\spadesuit$} (m-2-2) 
(m-2-2) edge node [above] {$\blacktriangle$} (m-1-3) 
(m-2-2) edge node [above] {$\blacktriangle$} (m-2-3) 
(m-2-2) edge node [above] {$\blacktriangle$} (m-3-3) 
(m-1-3) edge node [above] {$\clubsuit$} (m-1-4) 
(m-2-3) edge node [above] {$\clubsuit$} (m-2-4) 
(m-3-3) edge node [above] {$\clubsuit$} (m-3-4) 
(m-1-4) edge node [above] {$\blacktriangledown$} (m-2-5)
(m-2-4) edge node [above] {$\blacktriangledown$} (m-2-5)
(m-3-4) edge node [above] {$\blacktriangledown$} (m-2-5)
(m-2-1) edge node [left] {$\blacklozenge$} (m-3-1);
\end{tikzpicture}
\end{center}
	\caption{Reparametrization of $\Vanderm_m$ via the DFT and LDU in floating point arithmetic. Legend: $\spadesuit$ = the DFT of $\Vanderm_m$ using the explicit formulas (\ref{zd:eq:V*DFT}); $\blacktriangle$ = the pivoted LDU (\ref{zd:eq:Cauchy-LDU}) of $\Vanderm_m\DFT$  by explicit forward stable updates of the Schur complements \cite{Demmel-99-AccurateSVD}; $\clubsuit$ = forward errors in the computed factors $\widetilde{L}$, $\widetilde{\Delta}$, $\widetilde{U}$, bouded as in (\ref{zd:eq:Cauchy-LDU-error-bound}); $\blacktriangledown$ = implicit representation of $\Vanderm_m\DFT$ as the product $\Pi_1^T \widetilde{L}\widetilde{\Delta}\widetilde{U}\Pi_2^T$; $\blacklozenge$ = direct computation with $\Vanderm_m$, using standard algorithms, produces backward error $E$ that is small in matrix norm, and the condition number is $\kappa_2(\Vanderm_m)$.}
\end{figure}

A Matlab implementation of the decomposition (\ref{zd:eq:Cauchy-LDU}) is provided in Algorithm \ref{zd:Alg:VAND-FFT-LDU} of the Appendix (\S \ref{S=Matlab-codes}). For detailed analysis we refer to \cite{Demmel-99-AccurateSVD}, \cite{demmel-koev-2006-PolynVandermonde}.

\subsection{Cauchy matrix based reconstruction (in the frequency domain)}
We now revise the relation (\ref{zd:eq:X-W-V}), and write $\X_m\Vanderm_m^{-1}$ as $(\X_m\DFT) (\Vanderm_m \DFT)^{-1}$. This mere insertion of the identity $\DFT \DFT^{-1}$ between $\X_m$ and $\Vanderm_m^{-1}$ allows a natural interpretation: the time trajectories of the observables (the rows of $\X_m$) have been mapped to the frequency domain, and the inverse Vandermonde matrix of the eigenvectors of $C_m$ has been changed with the inverse of the generalized Cauchy matrix $\mathcal{D}_1\Cauchy \mathcal{D}_2$ (the eigenvectors of $\DFT^* C_m\DFT$); see Remark \ref{zd:REM:Krylov-DFT}.


Following \S \ref{SSS=LDU(Cauchy)}, compute the LDU decomposition with complete pivoting $\Pi_1 (\mathcal{D}_1 \Cauchy \mathcal{D}_2)\Pi_2 = L \Delta U$, and then apply $\Vanderm_m^{-1}$ through backward and forward substitutions, 
\begin{equation}\label{zd:eq:W-via-DFT}
\widehat{W}_m = (((((\X_m\DFT) \Pi_2) U^{-1})\Delta^{-1}) L^{-1})\Pi_1 . 
\end{equation}
The implementation of this formula depends on a particular software tool. For the reader's convenience, in Algorithm \ref{zd:Alg:inv(V)-via-DFT} we show a simple Matlab version,\footnote{Note that we have adapted post-multiplication by $\DFT$ to the Matlab's definition of the functions \texttt{fft()}, \texttt{ifft()}.} where the function \texttt{Vand\_DFT\_LDU()} computes the decomposition (\ref{zd:eq:Cauchy-LDU}).

\noindent This is obviously more complicated than ``backslashing'' in Matlab (i.e. $\widehat{W}_m=\X_m/\Vanderm_m$) and not as efficient as the fast Vandermonde inversion techniques such as the Bj\"{o}rck-Pereyra type algorithms that solve single Vandermonde system with $O(m^2)$ complexity. 
But, Algorithm \ref{zd:Alg:inv(V)-via-DFT}, when combined with Algorithm \ref{zd:Alg:VAND-FFT-LDU}, provides superior accuracy, independent of the distribution of the $\lambda_i$'s.

Besides ill--conditioning, one general difficulty with Vandermonde matrices is that the powers $\lambda_i^j$ may spread, in absolute value, hundreds of orders of magnitude and that underflows and overflows may cause irreparable exceptions in machine arithmetic. An interesting salient feature of the transformation (\ref{zd:eq:V*DFT}) is that the powers of $\lambda_i$ are changed into $1/(\lambda_i - \omega^{j-1})$ where the only powers are those of the primitive $m$th root of unity $\omega$; the powers of $\lambda_i$ are extracted in the form ${(\lambda_i^m-1)}/{\sqrt{m}}$ on the diagonal of the scaling matrix $\mathcal{D}_1$, where they can be kept until the very end of the computation. Namely, since our goal is to compute $\widehat{W}_m=\X_m \Vanderm_m^{-1}$ and its column norms, we can rephrase the previous formulas as  
$$
\widehat {W}_m = \left[ (\X_m \DFT) (\Cauchy\mathcal{D}_2)^{-1}\right] \mathcal{D}_1^{-1}
$$ 
and compute $\widetilde{W}_m = \left[ (\X_m \DFT) (\Cauchy\mathcal{D}_2)^{-1}\right]$ by using the LDU of $\Cauchy\mathcal{D}_2$, instead of (\ref{zd:eq:Cauchy-LDU}), and the forward and backward substitutions analogously to (\ref{zd:eq:W-via-DFT}). The diagonal entries of $\mathcal{D}_1^{-1}$ will then be assimilated as multiplicative factors in the
computation of $\mathfrak{a}_j$, see (\ref{zd:eq:Wm}).
(Analogously, we can write $\widetilde{W}_m=(\X_m \DFT \mathcal{D}_2^*) \Cauchy^{-1}$ and invert $\mathcal{\Cauchy}$ via its pivoted LDU, but this makes not a big difference since $\mathcal{D}_2$ is unitary.)
This modification can be easily implemented by minor modifications in Algorithm \ref{zd:Alg:inv(V)-via-DFT} and \ref{zd:Alg:VAND-FFT-LDU}; we omit the details for the sake of brevity.


\subsubsection{SVD based reconstruction.  Regularization}
Furthermore, based on the decomposition (\ref{zd:eq:Cauchy-LDU}), computed to high accuracy (\ref{zd:eq:Cauchy-LDU-error-bound}), we can compute accurate SVD of the product $L\Delta U$, $L\Delta U = \Omega \Sigma\Theta^*$, which gives an accurate SVD of $\Vanderm_m$,
$
\Vanderm_m = (\Pi_1^T \Omega) \Sigma (\DFT\Pi_2\Theta)^* \equiv \mathcal{U}\Sigma \mathcal{V}^*.
$
For details of the algorithm we refer to \cite{drm-98-psvd}, \cite{dgesvd-99}, \cite{Demmel-99-AccurateSVD}, \cite{drm-ves-VW-1}, \cite{drm-ves-VW-2}.

Now, instead of computing $\X_m\Vanderm_m^{-1}$ as a solution of linear system of equation, with the inverse $\Vanderm_m^{-1}=\mathcal{V}\Sigma^{-1}\mathcal{U}^*$, we change the framework into a regularized LS solution.
Let $\varphi_j\geq 0$ be filter factors, and 
\begin{equation}
\Sigma_\varphi^\dagger = \mathrm{diag}(\varphi_1 \frac{1}{\sigma_1},\ldots, \varphi_m\frac{1}{\sigma_m}).
\end{equation}
For instance, the commonly used regularization is
\begin{equation}
\Sigma_\varphi^\dagger = \mathrm{diag}(\frac{\sigma_1}{\sigma_1^2+\eta^2},\ldots, \frac{\sigma_m}{\sigma_m^2+\eta^2}),\;\; \mbox{where}\;\;\varphi_i = \frac{\sigma_i^2}{\sigma_i^2+\eta^2}
\end{equation}
Then we use the approximation
\begin{equation}
\X_m\Vanderm_m^{-1} \approx \X_m \mathcal{V}\Sigma_\varphi^\dagger\mathcal{U}^* = (\X_m\DFT\Pi_2) \Theta \Sigma_\varphi^\dagger\mathcal{U}^* \equiv (\X_m\DFT\Pi_2) \Theta \Sigma_\varphi^\dagger \Omega^* \Pi_1 .
\end{equation}
The key is that we have an accurate SVD even with extremely large $\sigma_1/\sigma_m$, and that with the tuning parameter $\eta\geq 0$ we can control the influence of the smallest singular values. This may be important in the case of noisy data -- an issue not considered in this paper.

%
%
%
%

%% file: SOURCES/Section_Numerical_Example_Case-Study.tex
\section{Numerical example: a case study}\label{SS:reconstruct:Num-Example}
To illustrate the preceding discussion, we use the simulation data of a 2D model obtained by depth averaging the Navier--Stokes equations for a shear flow in a thin layer of electrolyte suspended on a thin lubricating layer of a dielectric fluid; see \cite{Tithof-2017JFM-Bif-Q2D-Kolmogorov}, \cite{Suri-2014PhFl-Velocity-Kolmogorov-flow} for more detailed description of the experimental setup and numerical simulations.\footnote{We thank Michael Schatz, Balachandra Suri, Roman Grigoriev and Logan Kageorge from the Georgia Institute of Technology for providing us with the data.}

The (scalar) vorticity field data consists of $n_t$ snapshots of dimensions $n_x\times n_y$; in this particular example $n_t=1201\equiv m+1$, $n_x=n_y=128$.  The $n_x\times n_y \times n_t$ tensor is matricized into $n_x\cdot n_y \times n_t$ matrix $(\f_1,\ldots, \f_{n_t})$, and $\X_m$ is of dimensions $16384\times 1200$. 

The computational schemes are tested with respect to reconstruction potential as follows: for a given snapshot $\f_i$, the representation (\ref{zd:eq:f_i-dmd}) is truncated by taking given number of modes with absolutely largest amplitudes $|\mathfrak{a}_j|$ (abbreviated as \emph{dominant modes}). We note that, in general, selecting most appropriate modes is a separate nontrivial problem, not considered here. 

\subsection{Ill-conditioning of $\Vanderm_m$ and diagonal scalings}\label{SS=Numex-scaling}
In the first experiment, we illustrate the problem caused by high condition number of $\Vanderm_m$, and also the effects of simple diagonal scalings. 
More precisely, we attempt reconstruction of the snapshots as outlined in \S \ref{zd:SSS:modal-repr-V}, using the modes computed by
\begin{enumerate}
	\itemsep0em
 \item inversion of the Vandermonde matrix by the backslash operator in Matlab ;
\item  inversion of the row scaled Vandermonde matrix by the backslash operator in Matlab: $\Vanderm_m = D_r \Vanderm_m^{(r)}$, $\widehat{W}_m = (\X_m (\Vanderm_m^{(r)})^{-1}) D_r^{-1}$, where $D_r=\mathrm{diag}(\|\Vanderm_m(i,:)\|)_{i=1}^m$ ;
\item inversion of the column scaled Vandermonde matrix by the backslash operator in Matlab: $\Vanderm_m = \Vanderm_m^{(c)} D_c$, $\widehat{W}_m = (\X_m D_c^{-1})(\Vanderm_m^{(c)})^{-1}$, where $D_c=\mathrm{diag}(\|\Vanderm_m(:,i)\|)_{i=1}^m$.
\end{enumerate}
We choose the scaling in the $\ell_2$ norm, i.e. $\|\cdot\|=\|\cdot\|_2$.
The relevant condition numbers, estimated using the Matlab's function \texttt{cond()} are as follows
\begin{equation}\label{zd:eq:V-VR-VC-conditions}
\texttt{cond}(\Vanderm_m) \approx 8.9\cdot 10^{76},\;\;
\texttt{cond}(\Vanderm_m^{(r)}) \approx 3.1\cdot 10^{7},\;\;
\texttt{cond}(\Vanderm_m^{(c)}) \approx 3.0\cdot {10^{21}}.
\end{equation}
We can also scale in $\|\cdot\|_\infty$ or $\|\cdot\|_1$ norm, with similar effect as in (\ref{zd:eq:V-VR-VC-conditions}). (It is known that this scaling that equilibrates the columns, or rows, is nearly optimal in the class of all diagonal scalings, see \cite{slu-69}.) This is instructive; we see that the condition number of $\Vanderm_m$ can indeed be high, much higher than $1/\roff$, and that certain diagonal scaling may reduce it enough to allow sufficiently accurate computations in machine working precision $\roff$ (here assumed to sixteen decimal digits).

In Figure \ref{zd:fig:REC-321-300-VRS}, we display reconstruction results for $\f_{321}$, using  $300$ dominant  modes. The snapshots are visualized using the contour plots over the 2D domain.
The results are consistent with (\ref{zd:eq:V-VR-VC-conditions}) -- with the condition number above $1/\roff$ we do not expect any reasonable accuracy. 

The key for the success of row scaling of $\Vanderm_m$ is that in this case the distribution of the Ritz values $\lambda_i$ is such that it allows improvement of the condition number by row scaling. Although this will not be the case in general, it provides a good case study example for numerical analysts. It should be noted that $\Vanderm_m^{-1} = (\Vanderm_m^{(r)})^{-1} D_r^{-1}$ ($\widehat{W}_m = (\X_m (\Vanderm_m^{(r)})^{-1}) D_r^{-1}$) is just a different scaling of the eigenvectors of $C_m$ (Ritz vectors of $\A$), that is compensated by the corresponding scaling of the amplitudes; this scaling is as an allowable transformation -- an invariant of the representation. 

On the other hand, the column scaling of $\Vanderm_m$ did not reduce the condition number enough to ensure accurate inversion, although the reduction was by more than fifty orders of magnitude. Also, this scaling is not interpretable as an invariant of the reconstruction; note that it rescales the snapshots.

It is interesting that in the case of taking all or almost all modes, the reconstruction by simply backslashing $\Vanderm_m$ provides perfect reconstruction, see Figure \ref{zd:fig:REC-321-1199-VRS}. (This can be explained by the fact that even an inaccurate solution of a linear system of equations may have small residual.)
\begin{figure}
	\begin{minipage}[c]{0.70\textwidth}
		\includegraphics[width=\textwidth, height=0.90\textwidth]{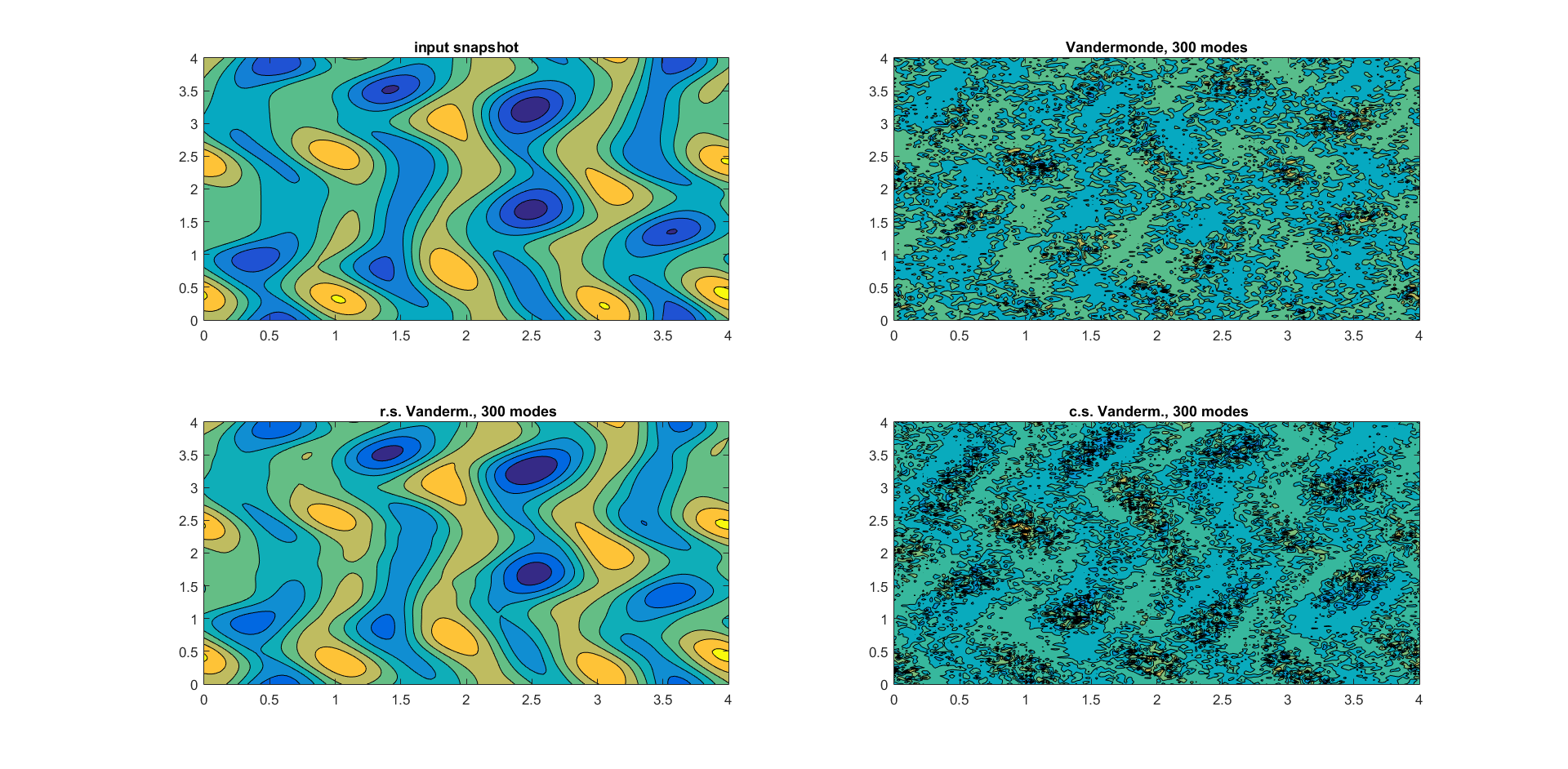}
	\end{minipage}\hfill
	\begin{minipage}[c]{0.30\textwidth}
		\caption{
\label{zd:fig:REC-321-300-VRS} Reconstruction of $\f_{321}$ using $300$ dominant modes. The linear systems are solved in Matlab using the backslash operator. Note how using the row scaled $\Vanderm_m^{(r)}$ improves the reconstruction (first plot in the second row), while backslashing the original $\Vanderm_m$ and the column scaled matrix $\Vanderm_m^{(c)}$  yields poor results (plots in the second column on the Figure). Similar effect is observed when reconstructing the other $\f_i$'s.
		} 
	\end{minipage}
\end{figure}
\begin{figure}[H]
	\begin{minipage}[c]{0.70\textwidth}
		\includegraphics[width=\textwidth, height=0.90\textwidth]{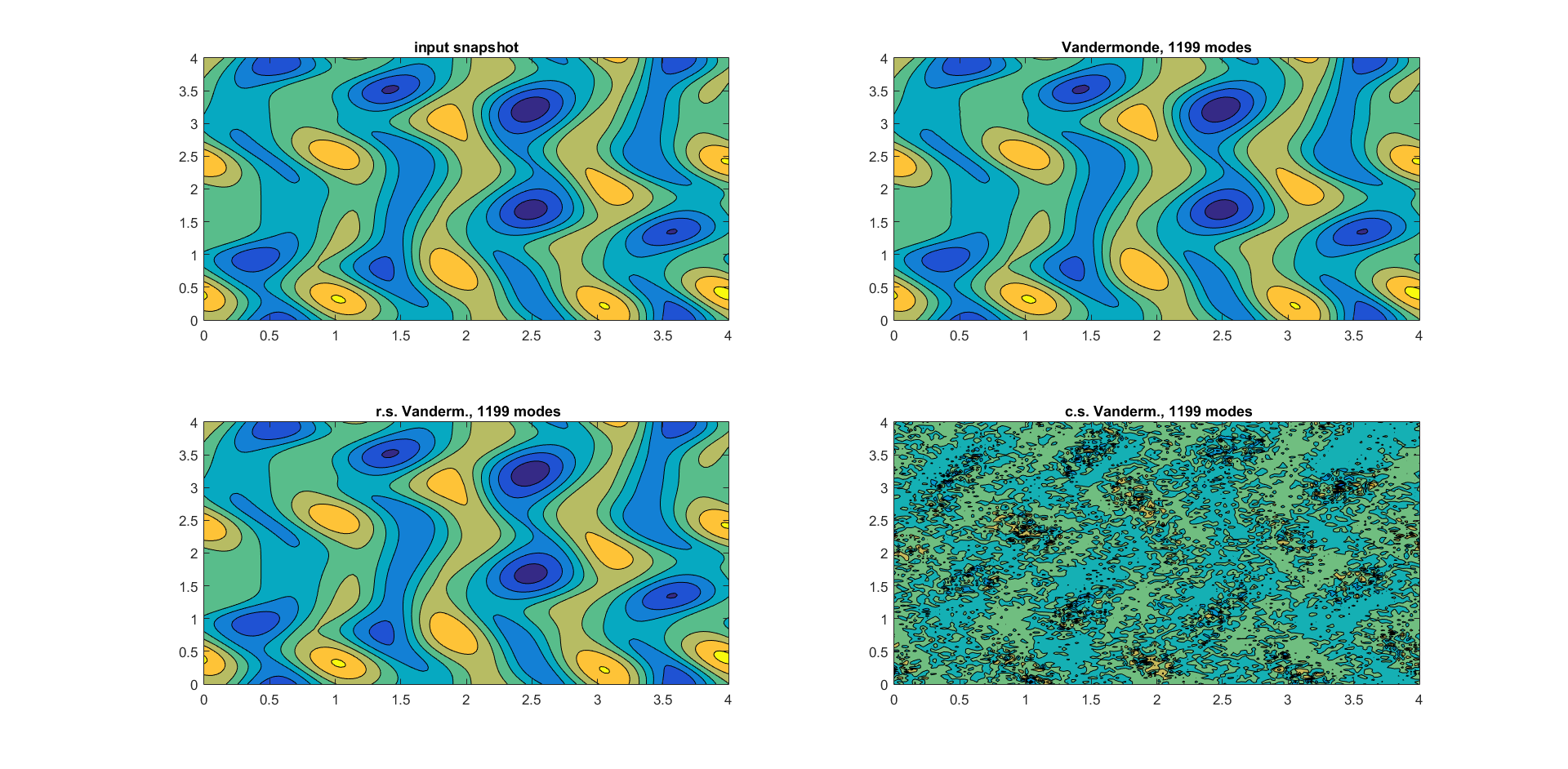}
	\end{minipage}\hfill
	\begin{minipage}[c]{0.30\textwidth}
		\caption{
			\label{zd:fig:REC-321-1199-VRS}  Reconstruction of $\f_{321}$ using $1199$ dominant modes. The linear systems are solved in Matlab using the backslash operator. Note that backslashing the original $\Vanderm_m$ and using larger number of modes resulted in good reconstruction.
		} 
	\end{minipage}
\end{figure}
%

For an analysis of optimally selected real nodes and the potential improvement of the condition number of Vandermonde matrices  by scaling we refer to \cite{Gautschi2011} with a \emph{caveat -- our nodes are complex and we have no luxury of selecting them}.  Our ultimate goal is to develop accurate computation (to the extent deemed possible by the perturbation theory) of $\widehat{W}_m = \X_m \Vanderm_m^{-1}$ independent of the distribution of the $\lambda_i$'s and independent of the range of their moduli, including the case of $\max_i|\lambda_i|/\min_i|\lambda_i|\gg 1$. \\

\subsection{Comparing Bj\"{o}rck-Pereyera, DMD and DFT based reconstruction}\label{SS=Numex-DMD-DFT}
Next, we do reconstruction using the following three methods:
\begin{enumerate}
\itemsep0em
\item companion matrix formulation with the Bj\"{o}rck-Pereyera method \cite{BjoP70} for Vandermonde systems. Although forward stable in the special case of real and ordered $\lambda_i$'s, this method may be very sensitive in the case of general complex $\lambda_i$'s and relatively large dimension $m$.
\item companion matrix formulation with the DFT and inversion of the Cauchy matrix, as described in \S \ref{zd:SS=DFT+Cauchy-solver}. Since $\DFT$ and $\mathcal{D}_2$ in (\ref{zd:eq:V*DFT}) are unitary, 
Algorithm \ref{zd:Alg:inv(V)-via-DFT} solves linear system with the matrix $\mathcal{D}_1\Cauchy = \Vanderm_m \DFT \mathcal{D}_2^*$ of condition number bigger than $10^{76}$. No additional scaling is used; we want to check the claim that such high condition number cannot spoil the result.
\item Schmid's DMD method as described in \S \ref{SS=Reconstruct_DMD}. Here we expect good reconstruction results, provided it is feasible for given data and the parameters. The SVD is not truncated because the spectral condition number of $\X_m$ is approximately $5.5\cdot 10^{10}$ ($\sigma_{\max}(\X_m)\approx 4.2\cdot 10^3$, $\sigma_{\min}(\X_m)\approx 7.5\cdot 10^{-8}$).
\end{enumerate}
The reconstruction results for $\f_{321}$, shown on Figure \ref{zd:fig:REC-321-30-ALL}, indicate that row scaled $\Vanderm_m$ based computation, as well as  DFT based and the Schmids' DMD are capable of reconstructing the snapshots with relatively small number of dominant modes.

\begin{figure}[H]
	\begin{minipage}[c]{0.70\textwidth}
		\includegraphics[width=\textwidth, height=0.90\textwidth]{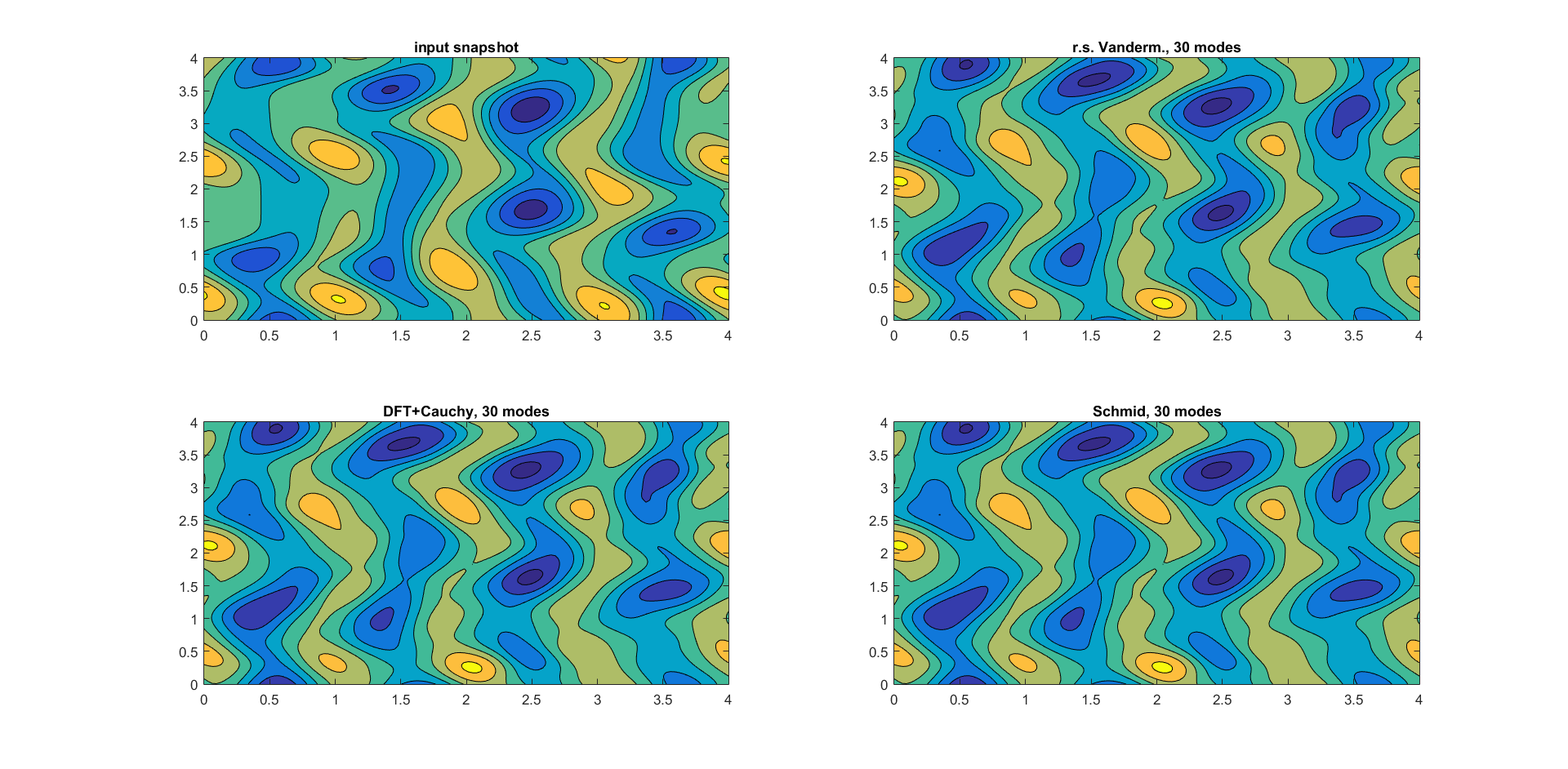}
	\end{minipage}\hfill
	\begin{minipage}[c]{0.30\textwidth}
		\caption{\label{zd:fig:REC-321-30-ALL} Reconstruction of $\f_{321}$ using $30$ dominant modes. The row scaled Vanermonde and the DFT+Cauchy inversion, as well as the Schmid's DMD reconstruction (second row) succeeded in reconstructing $\f_{321}$ using $30$ modes with dominant amplitudes.
		} 
	\end{minipage}
\end{figure}

The  Bj\"{o}rck-Pereyera method was not successful, and increasing the number of modes brings no improvement, see Figure \ref{zd:fig:REC-321-300-CSBP}.

\begin{figure}[H]
	\begin{minipage}[c]{0.70\textwidth}
		\includegraphics[width=\textwidth, height=0.90\textwidth]{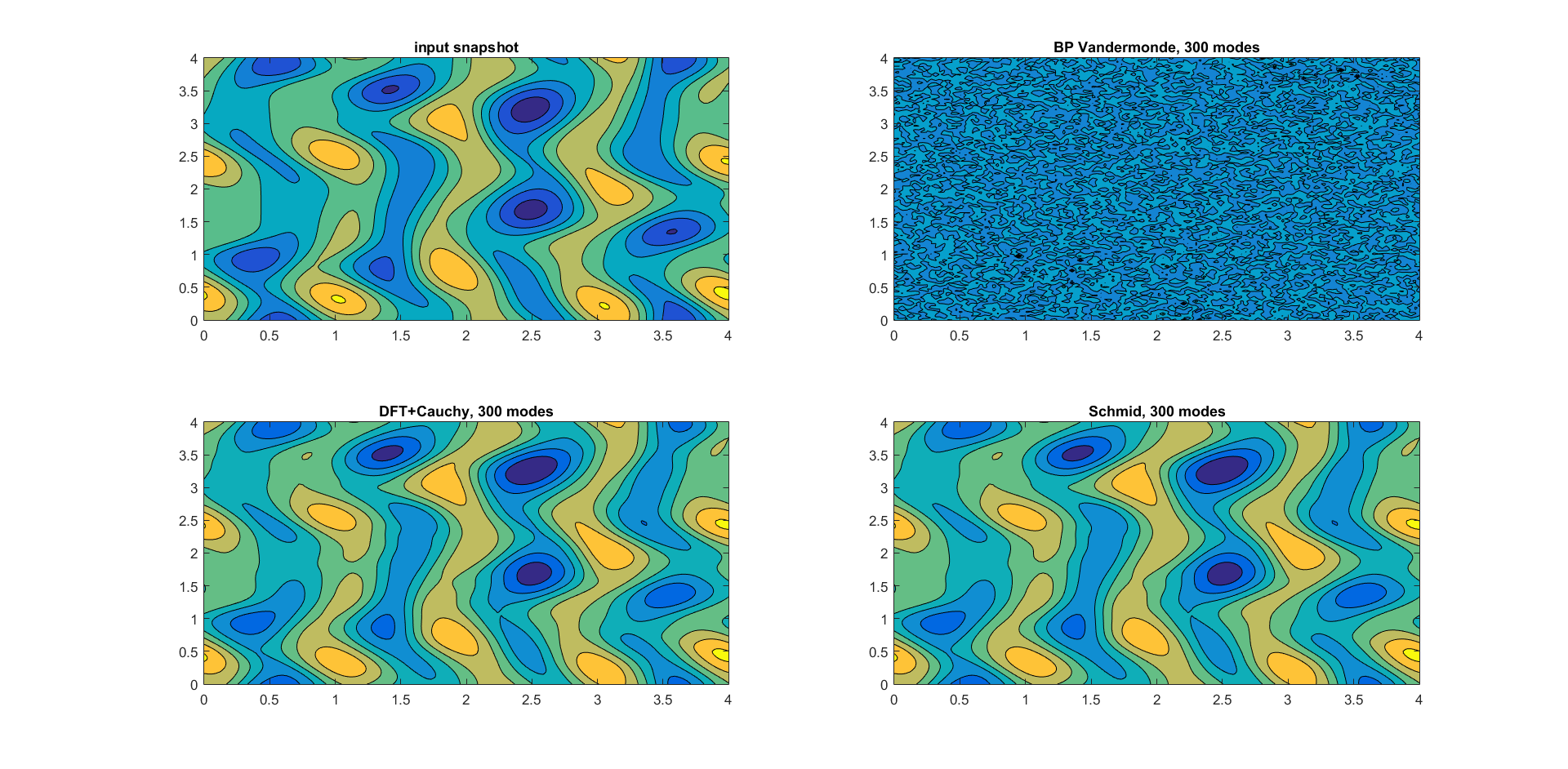}
	\end{minipage}\hfill
	\begin{minipage}[c]{0.30\textwidth}
		\caption{\label{zd:fig:REC-321-300-CSBP} Reconstruction of $\f_{321}$ using $300$ dominant modes. Bj\"{o}rck-Pereyera method (second plot in the first row) failed to produce any useful data. The DFT+Cauchy inversion and the Schmid's DMD reconstruction (second row) succeeded in reconstructing $\f_{321}$ pretty much using $300$ modes with dominant amplitudes.
		} 
	\end{minipage}
\end{figure}
Increasing the number of modes in the reconstruction further may or may not not bring improvement, as seen in Figure \ref{zd:fig:REC-1111-1075}. Ar first, it may come as a surprise that the SVD based DMD starts failing after taking more and more nodes. On the other hand, we must take into account that the left singular vectors of the matrix $\X_m$ as well as the eigenvectors of the Rayleigh quotient $S_k$ (lines 1. and 5. in the DMD Algorithm \ref{zd:ALG:DMD}) may be computed inaccurately as their sensitivity depends on the condition numbers of $\X_m$ and $S_k$ (line 4.) and the gaps in the singular values and the eigenvalues, respectively.

\begin{figure}[H]
	\begin{minipage}[c]{0.70\textwidth}
		\includegraphics[width=\textwidth, height=0.90\textwidth]{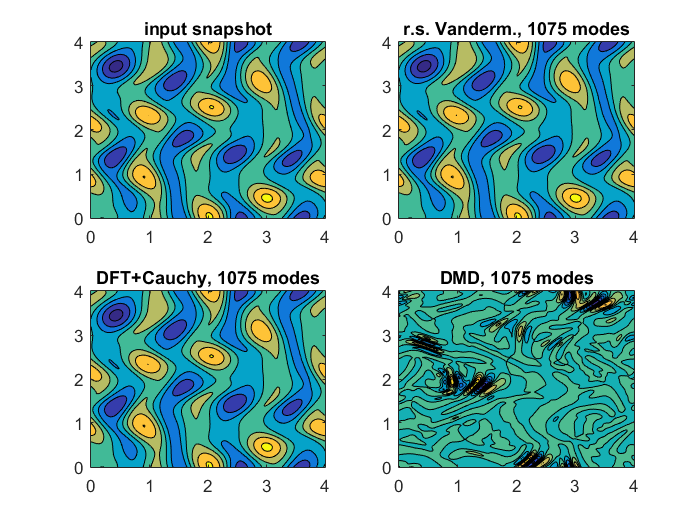}
	\end{minipage}\hfill
	\begin{minipage}[c]{0.30\textwidth}
		\caption{\label{zd:fig:REC-1111-1075} Reconstruction of $\f_{1111}$ using $1075$ dominant modes.  The DFT+Cauchy inversion did well, and the Schmid's DMD reconstruction (second row) unfortunately failed. (With $1074$ modes DMD performed well, but starting with $1075$ all reconstruction failed, including the one with all $1200$ modes.)
		} 
	\end{minipage}
\end{figure}

\begin{remark}
The result shown in Figure \ref{zd:fig:REC-1111-1075}, showing  unsuccessful reconstruction by the SVD based DMD starting at mode $1075$ seems surprising. On the other hand, in DMD, the modes are computed from the product of the left singular vector matrix of the snapshots and the eigenvectors of the Rayleigh quotient; both sets of vectors are sensitive to small gaps in the spectrum and may not be computed to high accuracy; recall Remark \ref{zd:REM:eig-vec-sens}. For the sensitivity of the singular vectors see \cite{RC-Li-Pert-Th-II-1998}. 	
\end{remark}


\subsection{Reconstruction in the QR compressed space}
	An efficient method of reducing the dimension of the ambient space is the QR--compressed scheme \cite[\S 5.2.1]{DDMD-SISC-2018}, which reduces $n$ to $m+1$
	via the QR factorization of $\X_{m+1} = ( \X_m , \f_{m+1})$. We briefly review the main idea; for more details and the general case of non-sequential data $\X_m$ and $\Y_m=\A\X_m$ we refer the reader to \cite{DDMD-SISC-2018}.
	
	In the QR factorization
	\begin{equation}\label{zd:eq:QRFm+1}
	\X_{m+1} = Q_f \begin{pmatrix} R_f \cr 0\end{pmatrix} = \widehat{Q}_f R_f,\;\;Q_f^* Q_f = I_n, \;\;\widehat{Q}_f=Q_f(:,1:m+1),
	\end{equation}
	where $R_f$ is $(m+1)\times (m+1)$ upper triangular, it holds that 
	 $\mathrm{range}(\widehat{Q}_f)\supseteq \mathrm{range}(\X_{m+1})$. (Clearly, $\mathrm{range}(\widehat{Q}_f)= \mathrm{range}(\X_{m+1})$ if and only if $\X_{m+1}$ is of full column rank.) Hence, we can reparametrize the data, and work in the new representation in the basis defined by the columns of $\widehat{Q}_f$. To that end, set $\Y_m=\A\X_m = (\f_2,\ldots, f_{m+1})$, $R_x=R_f(:,1:m)$, $R_y=R_f(:,2:m+1)$. Then 
	\begin{equation}\label{zd:eq:FXY-compress}
	\X_m = \widehat{Q}_f R_x,\;\;\Y_m =\widehat{Q}_f R_y \; ;\;\;
	R_f = \left( \begin{smallmatrix} \times & \divideontimes & \divideontimes & \divideontimes & \div \cr & \divideontimes & \divideontimes & \divideontimes & \div \cr 
	&  & \divideontimes & \divideontimes & \div \cr 
	&   &  & \divideontimes & \div \cr
	&   &   &   & \div\end{smallmatrix}\right),\;\;
	R_x = \left( \begin{smallmatrix} 
	\times & \divideontimes & \divideontimes & \divideontimes \cr 
	& \divideontimes & \divideontimes & \divideontimes \cr
	&        & \divideontimes & \divideontimes \cr
	&        &        & \divideontimes \cr
	&        &        &   0
	\end{smallmatrix}\right),\;\;
	R_y = \left(\begin{smallmatrix} 
	\divideontimes & \divideontimes & \divideontimes & \div \cr
	\divideontimes & \divideontimes & \divideontimes & \div \cr
	& \divideontimes & \divideontimes & \div \cr
	&      & \divideontimes & \div \cr
	&      &      & \div
	\end{smallmatrix}\right) ,
	\end{equation}
	and, in the basis of the columns of $\widehat{Q}_f$, we can identify $\X_m\equiv R_x$, $\Y_m\equiv R_y$, i.e. we can think of $\X_m$ and $\Y_m$ as $m$ snapshots in an $(m+1)$ dimensional space. 

	Further, the coefficients $\cb=(c_i)_{i=1}^m$ of the companion matrix $C_m$ are computed as
	$$
	\cb = R_x^\dagger R_y(:,m) = R_x(1:m,1:m)^{-1}R_y(1:m,m) \equiv
	R_f(1:m,1:m)^{-1}R_f(1:m,m+1),
	$$
	where the explicit use of $R_x^{-1}$ assumes the full rank case. Numerically, we would solve the LS problem $\| R_x \cb - R_y(:,m)\|_2\rightarrow\min$ by an appropriate method (including refinement).

In this representation, if we set $\widehat{\f}_i\equiv R_x(:,i)$, and if we have a computed reconstruction $\widetilde{\widehat{\f}}_i \approx \widehat{\f}_i$, then $\widehat{Q}_f \widetilde{\widehat{\f}}_i \approx \f_i$, with the error
$
\| \widehat{Q}_f \widetilde{\widehat{\f}}_i - \widehat{Q}_f \widehat{\f}_i\|_2 = 
\|\widetilde{\widehat{\f}}_i -  \widehat{\f}_i\|_2
$. Hence we can run the entire process (including the initial DMD) in the basis $\widehat{Q}_f$, and switch to the original representation at the very end. 
Since the factorization (\ref{zd:eq:QRFm+1}) is available in high performance software implementations, the overall procedure is more efficient in particular in the cases when $m\ll n$. 

Let us now from try this scheme using the test data from \S \ref{SS=Numex-scaling} and \S \ref{SS=Numex-DMD-DFT}. As expected, the compressed scheme computes considerably faster. The numerical results are similar, with some minor variations. To illustrate, we again reconstruct $\f_{321}$ and $\f_{1111}$.

 The results shown in Figure \ref{zd:fig:REC-321-300-VRS-QR}, when compared to Figure \ref{zd:fig:REC-321-300-VRS}, seem to indicate that pure Vandermonde inversion, while still performing poorly,  shows some improvement. This is probably due to the fact that the Vandermonde inverse is applied to a triangular matrix; we omit the discussion for the sake of brevity. (The DMD and the DFT+Cauchy algorithms produced results similar to the ones in Figure \ref{zd:fig:REC-321-30-ALL}.)

\begin{figure}[H]
	\begin{minipage}[c]{0.70\textwidth}
		\includegraphics[width=\textwidth, height=0.90\textwidth]{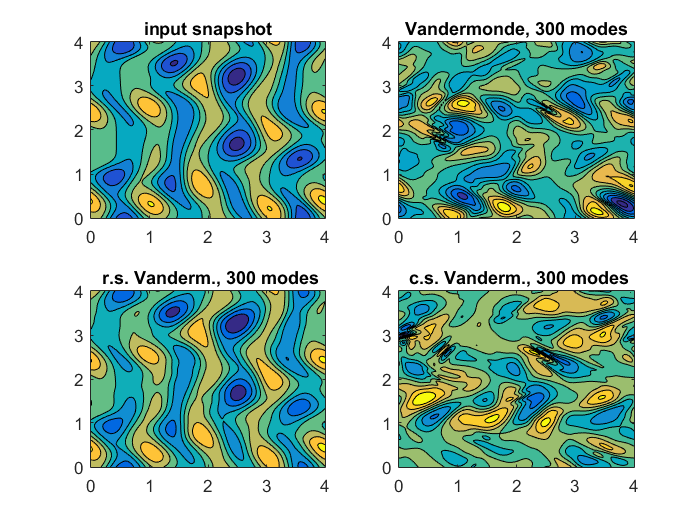}
	\end{minipage}\hfill
	\begin{minipage}[c]{0.30\textwidth}
		\caption{\label{zd:fig:REC-321-300-VRS-QR} Reconstruction of $\f_{321}$ using $300$ dominant modes in the basis $\widehat{Q}_f$. Compare with Figure \ref{zd:fig:REC-321-300-VRS}.  Similar effect is observed when reconstructing the other $\f_i$'s.
		} 
	\end{minipage}
\end{figure}

With the snapshot $\f_{1111}$ (see Figure \ref{zd:fig:REC-1111-1075}), the QR compressed DMD started failing at the number of $1073$ dominant modes; the DFT+Cauchy method performed well, as before.

%% file: SOURCES/Section_Freq_Domain_Reconstruct.tex
\section{Snapshot reconstruction: theoretical insights}\label{S=Snapshot-rec-theory}
\noindent It is desirable to have a representation analogous to (\ref{eq:reconstruction-formula}), but with smaller number of eigenmodes and with as small as possible representation error. Suppose that we use $\ell$ modes, $\ell < m$, and that the Ritz pairs are so enumerated that the selected ones are $(\lambda_1, z_1), \ldots, (\lambda_\ell, z_\ell)$. We do not specify how the pairs $(\lambda_i,z_i)$ have been  computed -- the $z_i$'s can be the Ritz vectors or e.g. the refined Ritz vectors \cite{DDMD-SISC-2018}. The selection of the particular  $\ell$ pairs can be guided e.g. by the sparsity promoting DMD \cite{jovschnicPOF14}.
Wanted are the coefficients $\alpha_1, \ldots, \alpha_{\ell}$ that minimize the least squares error over all snapshots:
\begin{equation}\label{eq:rec-error-min}
\sum_{i=1}^m \| \f_i - \sum_{j=1}^\ell z_j \alpha_j \lambda_j^{i-1}\|_2^2 \longrightarrow \min .
\end{equation}
If we set $Z_\ell = (z_1 , \ldots, z_\ell)$, $\vec{\bfalpha}=(\alpha_1,\ldots, \alpha_\ell)^T$, $\Delta_\bfalpha=\mathrm{diag}(\vec{\bfalpha})$, $\Lambda_j=(\lambda_1^{j-1}, \ldots, \lambda_\ell^{j-1})^T$, and $\Delta_{\Lambda_j}=\mathrm{diag}(\Lambda_j)$
then the objective (\ref{eq:rec-error-min}) can be written as 
\begin{equation}\label{eq:Omega(alpha)}
\Omega^2({\bfalpha}) \equiv \|  \X_m - Z_{\ell} \Delta_{\bfalpha} \begin{pmatrix} \Lambda_1 & \Lambda_2 & \ldots & \Lambda_{m}\end{pmatrix} \|_F^2 \longrightarrow\min.
\end{equation}
If we compute the economy size (tall) QR factorization $Z_\ell = Q R$ and define projected snapshots $\g_i = Q^* \f_i$, then the LS problem can be compactly written as
\begin{equation}\label{zd:eq:g-S*alpha}
\|  \vec{\g} -
S \vec{\bfalpha} \|_2\longrightarrow\min,\;\;\mbox{where}\;\;\vec{\g}= \begin{pmatrix} \g_1 \cr \vdots \cr \g_m\end{pmatrix},\;\; S = (I_m\otimes R) \begin{pmatrix} 
\Delta_{\Lambda_1} \cr \vdots \cr  \Delta_{\Lambda_m}\end{pmatrix} \equiv 
\begin{pmatrix} 
R \Delta_{\Lambda_1} \cr \vdots \cr  R \Delta_{\Lambda_m}\end{pmatrix} .
\end{equation}
The  normal equations approach \cite{jovschnicPOF14} allows efficient computation of $\vec{\bfalpha}=S^\dagger \vec{\g}$ due to the particular structure of $S$. 
Recently, \cite{DMD-LS-reconstruct-2018} proposed a corrected semi-normal solution and a QR factorization based method. 
 
In this section, our interests are of theoretical nature. We explore other choices of the pseudo--inverse based solution, different from the Moore--Penrose solution $S^\dagger \vec{\g}$, in an attempt to establish a connection between the numerical linear algebra framework and the GLA.

\subsection{Minimization in a $Z_\ell$--induced norm}
\noindent The  reflexive g--inverse \cite[Definition 2.4]{rao-mitra-1972} of $S$,
$$
S^{-} = \begin{pmatrix} 
\Delta_{\Lambda_1} \cr \vdots \cr  \Delta_{\Lambda_m}\end{pmatrix}^{\dagger}(I\otimes R^{-1}),
$$
allows interesting explicit formulas that reveal a  relation to the GLA \cite{Mohr:2014wm,Mezic:2013ei}.
Indeed, if we define $\vec{\bfalpha}_\star = S^-\vec{\g}$, then:
\begin{eqnarray}
\vec{\bfalpha}_\star &=& \begin{pmatrix} 
\Delta_{\Lambda_1} \cr \vdots \cr  \Delta_{\Lambda_m}\end{pmatrix}^{\dagger}(I\otimes R^{-1})\begin{pmatrix} \g_1 \cr \vdots \cr \g_m\end{pmatrix}= (\sum_{k=1}^m \Delta_{\Lambda_k}^* \Delta_{\Lambda_k} )^{-1} \sum_{i=1}^m \Delta_{\Lambda_i}^* (R^{-1}\g_i) \label{eq:S-*g}\\
&=&
\left(\begin{smallmatrix} \sum_{k=1}^m |\lambda_1|^{2(k-1)} & 0 & \ldots & 0 \cr
0 & \sum_{k=1}^m |\lambda_2|^{2(k-1)} & \ddots & \vdots \cr
\vdots & \ddots & \ddots & 0 \cr
0 & \ldots & 0 & \sum_{k=1}^m |\lambda_{\ell}|^{2(k-1)}\end{smallmatrix}\right)^{-1} 
\sum_{i=1}^m \Delta_{\Lambda_i}^* (R^{-1}\g_i) \\
&=& \sum_{i=1}^m \begin{pmatrix} \frac{\overline{\lambda}_1^{i-1}}{\sum_{k=1}^m |\lambda_1|^{2(k-1)}} & 0 & \cdot & 0 \cr 
0 & \frac{\overline{\lambda}_2^{i-1}}{\sum_{k=1}^m |\lambda_2|^{2(k-1)}} & \cdot & \cdot \cr
\cdot & \cdot & \cdot & 0 \cr
0 & \cdot & 0 & \frac{\overline{\lambda}_{\ell}^{i-1}}{\sum_{k=1}^m |\lambda_{\ell}|^{2(k-1)}}\end{pmatrix} \begin{pmatrix} (R^{-1}\g_i)_1 \cr 
(R^{-1}\g_i)_2 \cr \vdots\cr (R^{-1}\g_i)_{\ell}\end{pmatrix} \\
&=& \sum_{i=1}^m \begin{pmatrix} \frac{\overline{\lambda}_1^{i-1}}{\sum_{k=1}^m |\lambda_1|^{2(k-1)}} (R^{-1}\g_i)_1 \cr 
\frac{\overline{\lambda}_2^{i-1}}{\sum_{k=1}^m |\lambda_2|^{2(k-1)}}(R^{-1}\g_i)_2 \cr \vdots\cr \frac{\overline{\lambda}_{\ell}^{i-1}}{\sum_{k=1}^m |\lambda_{\ell}|^{2(k-1)}}(R^{-1}\g_i)_{\ell}\end{pmatrix} 
= 
\begin{pmatrix} \sum_{i=1}^m\frac{\overline{\lambda}_1^{i-1}}{\sum_{k=1}^m |\lambda_1|^{2(k-1)}} (R^{-1}\g_i)_1 \cr 
\sum_{i=1}^m\frac{\overline{\lambda}_2^{i-1}}{\sum_{k=1}^m |\lambda_2|^{2(k-1)}}(R^{-1}\g_i)_2 \cr \vdots\cr \sum_{i=1}^m\frac{\overline{\lambda}_{\ell}^{i-1}}{\sum_{k=1}^m |\lambda_{\ell}|^{2(k-1)}}(R^{-1}\g_i)_{\ell}\end{pmatrix} . \label{eq:optimal-reconstruction-weights-alpha-star}
\end{eqnarray}
Although this $\vec{\bfalpha}_\star$  does not minimize (\ref{eq:Omega(alpha)}), nor (\ref{zd:eq:g-S*alpha}),  it follows from the theory of generalized inverses that it does have some interesting properties. 
\begin{remark}
Recall that, by definition, $S^-$ satisfies
\begin{equation}
S S^- S = S,\;\;S^- S S^- = S^-  \mbox{and}\;\; S^- S = (S^- S)^*\ .
\end{equation}
In fact, since $S S^- \neq (S S^-)^*$, we have in general that $S^- \neq S^\dagger$, so $\vec{\bfalpha}_\star \neq S^\dagger\vec{\g}$.
\end{remark}
\begin{proposition}\label{zd:PROP:M-LS}
	Let $M=I \otimes (RR^*)^{-1}$, $(x,y)_M = y^* Mx$, and let $\|x\|_M = \sqrt{x^*Mx}$. Then $\vec{\bfalpha}_\star$ is the minimum $\|\cdot\|_2$--norm solution of the weighted least squares problem 
	\begin{equation}
	\|\vec{g} - S\vec{\bfalpha}\|_M \longrightarrow \min .
	\end{equation}	
\end{proposition}
\begin{proof}
	The matrix $M$ is obviously positive definite and $\|\cdot\|_M$ is well defined norm. The claim is a direct corollary of \cite[Theorem 3.3]{rao-mitra-1972}, because one can easily establish that 
	$$
	(SS^-)^* = (I \otimes (RR^*)^{-1}) (SS^-) (I \otimes (RR^*)) = M (SS^-)M^{-1}.
	$$	
\noindent The claim can be derived also by brute force calculation. Since 
$$
M = (I\otimes R^{-*})(I\otimes R^{-1}) \equiv LL^*, \;\;L\equiv I\otimes R^{-*},
$$
is the Cholesky factorization and $\|x\|_M = \| (I\otimes R^{-1}) x\|_2$, we have
\begin{eqnarray}
\norm*{ \begin{pmatrix} \g_1 \cr \vdots \cr \g_m\end{pmatrix} - \begin{pmatrix} 
R \Delta_{\Lambda_1} \cr \vdots \cr  R  \Delta_{\Lambda_m}\end{pmatrix} \vec{\bfalpha} }_M^2 
&=& \norm*{ (I\otimes R^{-1}) \left[\begin{pmatrix} \g_1 \cr \vdots \cr \g_m\end{pmatrix} - \begin{pmatrix} 
R \Delta_{\Lambda_1} \cr \vdots \cr  R  \Delta_{\Lambda_m}\end{pmatrix} \vec{\bfalpha} \right] }_2^2 \\
&=& \sum_{j=1}^m \|R^{-1}\g_j-\Delta_{\Lambda_j}\vec{\bfalpha}\|_2^2 = \sum_{j=1}^m \|\g_j-R \Delta_{\Lambda_j}\vec{\bfalpha}\|_M^2.
\end{eqnarray}
However, the framework of the theory of matrix generalized inverses gives deeper insights. 
\end{proof}

\begin{remark}
		If 	$Z_\ell^* Z_\ell = I_\ell$ (e.g. if $\A$ is normal, and we compute an orthonormal set of Ritz vectors), then $R$ is diagonal and unitary and $S^\dagger=S^-$.
\end{remark}

\subsection{LS reconstruction in the frequency domain}\label{S=DFT-reconstruct}
We now repeat the reconstruction procedure, but in the frequency domain, following the methodology of \S \ref{zd:SS=DFT+Cauchy-solver}. That is, instead of $n$ time observable trajectories sampled at $m$ time-equidistant points (the rows of $\X_m$) we have their DFT transforms, i.e. $\X_m\DFT$, as the raw input.
We assume that the eigenvalues are simple and are enumerated so that the selected ones are $\lambda_1,\ldots, \lambda_\ell$; the QR factorization of the selected modes is $Z_{\ell}=QR$.

If we post-multiply  the expression in (\ref{eq:Omega(alpha)}) by the unitary DFT matrix $\DFT$, the objective becomes 
\begin{equation}\label{eq:Omega(alpha)-DFT}
\Omega(\vec{\bfalpha}) \equiv \| \X_m \DFT - Z_{\ell} \Delta_{\bfalpha} \begin{pmatrix} \Lambda_1 & \Lambda_2 & \ldots & \Lambda_{m}\end{pmatrix}\DFT \|_F^2 \longrightarrow\min\;\;\mbox{as the function of $\vec{\bfalpha}$},
\end{equation}
where $\begin{pmatrix} \Lambda_1 & \Lambda_2 & \ldots & \Lambda_{m}\end{pmatrix}\DFT$ is the DFT of the $\ell \times m$ submatrix $\Vanderm_{\ell,m}$ of $\Vanderm_m$.
Because of the structure of $\Vanderm_{\ell,m}\DFT$, we distinguish two cases.

\subsubsection{{Case 1}: $\prod_{i=1}^\ell (\lambda_i^ m - 1) \neq 0$}
If none of the $\lambda_i$'s equals an $m$th root of unity, then (see (\ref{zd:eq:V*DFT}))
\begin{displaymath}
\begin{pmatrix} \Lambda_1 & \Lambda_2 & \ldots & \Lambda_{m}\end{pmatrix}\DFT = 
\widehat{\mathcal{D}}_1 \widehat{\Cauchy} {\mathcal{D}}_2
\end{displaymath}
where $\widehat{\cdot}$ denotes an $\ell\times m$ submatrix from the relation (\ref{zd:eq:V*DFT}). In particular, $(\widehat{\mathcal{D}}_1)_{ii}=(\lambda_i^m-1)/\sqrt{m}\neq 0$, $(\widehat{\Cauchy})_{ij}=1/(\lambda_i-\overline{\omega}^{j-1})$, and ${\mathcal{D}}_2$ is unitary diagonal matrix. Introduce column partition
\begin{equation}\label{eq:C-hat-partition}
\widehat{\Cauchy} = \begin{pmatrix} \widehat{\Cauchy}_1 & \ldots & \widehat{\Cauchy}_m\end{pmatrix},\;\;
\widehat{\Cauchy}_j = \begin{pmatrix} {\displaystyle \frac{1}{\lambda_1-\overline{\omega}^{j-1}}}\cr \vdots \cr 
{\displaystyle \frac{1}{\lambda_{\ell}-\overline{\omega}^{j-1}}}
\end{pmatrix},\;\;
\Delta_{\Cauchy_j} = \begin{pmatrix} {\displaystyle \frac{1}{\lambda_1-\overline{\omega}^{j-1}}} & 0 & \cdot & 0 \cr 
0 & {\displaystyle \frac{1}{\lambda_2-\overline{\omega}^{j-1}}} & \cdot & \cdot \cr
\cdot & \cdot & \cdot & 0 \cr
0 & \cdot & 0 & {\displaystyle \frac{1}{\lambda_{\ell}-\overline{\omega}^{j-1}}}\end{pmatrix} .
\end{equation}
Let $\widehat{\X}_m = \X_m \DFT \mathcal{D}_2^*=(\widehat{\f}_1, \ldots, \widehat{\f}_m)$, $\Delta_{\bfbeta}=\Delta_{\bfalpha}\widehat{\mathcal{D}}_1$. Then
\begin{eqnarray}
\Omega(\vec{\bfalpha}) &=& \|\widehat{\X}_m -  Z_\ell \Delta_{\bfbeta} \widehat{\Cauchy}\|_F^2 = \| (\widehat{\f}_1, \ldots, \widehat{\f}_m) - Z_\ell \Delta_{\bfbeta}\begin{pmatrix} \widehat{\Cauchy}_1 & \ldots & \widehat{\Cauchy}_m\end{pmatrix}\|_F^2 \label{eq:Omega-beta-1}\\
&=& 
\norm*{ \begin{pmatrix} \widehat{\f}_1 \cr \vdots \cr \widehat{\f}_m\end{pmatrix} - \begin{pmatrix} Z_{\ell}\Delta_{\bfbeta}\Cauchy_1 \cr \vdots \cr Z_{\ell}\Delta_{\bfbeta}\Cauchy_m\end{pmatrix}}_2^2 =\mbox{\framebox{use: $\Delta_{\bfbeta}\Cauchy_i = \Delta_{\Cauchy_i}\vec{\bfbeta}$}} 
= \norm*{ \begin{pmatrix} \widehat{\f}_1 \cr \vdots \cr \widehat{\f}_m\end{pmatrix} - \begin{pmatrix} Z_{\ell}\Delta_{\Cauchy_1} \cr \vdots \cr Z_{\ell}\Delta_{\Cauchy_m}\end{pmatrix} \vec{\bfbeta}}_2^2 .\nonumber
\end{eqnarray}
If we set  $Q^*\widehat{\f}_i = \widehat{\g}_i$, then, using the reflexive g--inverse as before,  
\begin{eqnarray*}
\vec{\bfalpha} &=& \widehat{\mathcal{D}}_1^{-1} \begin{pmatrix} 
\Delta_{\Cauchy_1} \cr \vdots \cr  \Delta_{\Cauchy_m}\end{pmatrix}^{\dagger}(I\otimes R^{-1})\begin{pmatrix} \widehat{\g}_1 \cr \vdots \cr \widehat{\g}_m\end{pmatrix}
=\widehat{\mathcal{D}}_1^{-1} (\sum_{k=1}^m \Delta_{\Cauchy_k}^* \Delta_{\Cauchy_k} )^{-1} \sum_{j=1}^m \Delta_{\Cauchy_j}^* (R^{-1}\widehat{\g}_j) \\
&=&
\widehat{\mathcal{D}}_1^{-1}
\left(\begin{smallmatrix} \sum_{k=1}^m \frac{1}{|\lambda_1 - \overline{\omega}^{k-1}|^2} & 0 & \ldots & 0 \cr
0 & \sum_{k=1}^m \frac{1}{|\lambda_2 - \overline{\omega}^{k-1}|^2} & \ddots & \vdots \cr
\vdots & \ddots & \ddots & 0 \cr
0 & \ldots & 0 & \sum_{k=1}^m \frac{1}{|\lambda_{\ell} - \overline{\omega}^{k-1}|^2}\end{smallmatrix}\right)^{-1} 
\sum_{j=1}^m \Delta_{\Cauchy_j}^* (R^{-1}\widehat{\g}_j) \\
&=& 
\widehat{\mathcal{D}}_1^{-1}\sum_{j=1}^m \begin{pmatrix} \frac{1}{\sum_{k=1}^m \frac{\overline{\lambda}_1 -{\omega}^{j-1}}{|\lambda_1-\overline{\omega}^{k-1}|^2}} (R^{-1}\widehat{\g}_j)_1 \\[5mm]
\frac{1}{\sum_{k=1}^m \frac{\overline{\lambda}_2 -{\omega}^{j-1}}{|\lambda_2-\overline{\omega}^{k-1}|^2}}(R^{-1}\widehat{\g}_j)_2 \cr \vdots\cr \frac{1}{\sum_{k=1}^m \frac{\overline{\lambda}_{\ell} -{\omega}^{j-1}}{|\lambda_{\ell}-\overline{\omega}^{k-1}|^2}}(R^{-1}\widehat{\g}_j)_{\ell}\end{pmatrix} 
= 
\widehat{\mathcal{D}}_1^{-1} \begin{pmatrix} \sum_{j=1}^m\frac{1}{\sum_{k=1}^m \frac{\overline{\lambda}_1 -{\omega}^{j-1}}{|\lambda_1-\overline{\omega}^{k-1}|^2}} (R^{-1}\widehat{\g}_j)_1 \\[5mm]
\sum_{j=1}^m\frac{1}{\sum_{k=1}^m \frac{\overline{\lambda}_2 -{\omega}^{j-1}}{|{\lambda}_2-\overline{\omega}^{k-1}|^2}}(R^{-1}\widehat{\g}_j)_2 \cr \vdots\cr \sum_{j=1}^m\frac{1}{\sum_{k=1}^m \frac{\overline{\lambda}_{\ell} -{\omega}^{j-1}}{|{\lambda}_{\ell}-\overline{\omega}^{k-1}|^2}}(R^{-1}\widehat{\g}_j)_{\ell}\end{pmatrix}\!\! .
\end{eqnarray*}
Hence, the components $\alpha_i$ of the LS solution via the reflexive g--inverse are
\begin{displaymath}
\bfalpha_i = \frac{\sqrt{m}}{\lambda_i^m-1} \sum_{j=1}^m
\frac{{\displaystyle \frac{1}{\overline{\lambda}_{i} -{\omega}^{j-1}}}}{{\displaystyle \sum_{k=1}^m \frac{1}{|{\lambda}_{i}-\overline{\omega}^{k-1}|^2}}}(R^{-1}\widehat{\g}_j)_{i} = \frac{\sqrt{m}}{{\displaystyle \prod_{p=1}^m (\lambda_i-\overline{\omega}^{p-1})}} \sum_{j=1}^m
\frac{{\displaystyle \frac{1}{\overline{\lambda}_{i} -{\omega}^{j-1}}}}{{\displaystyle \sum_{k=1}^m \frac{1}{|{\lambda}_{i}-\overline{\omega}^{k-1}|^2}}}(R^{-1}\widehat{\g}_j)_{i} .
\end{displaymath}

\begin{remark}
It remains an interesting open question whether these formulas have an appropriate interpretation in a wider context. Also, its ingredients such as 
$R^{-1}\widehat{\g}_j$ or the harmonic mean of squared distances from $\lambda_i$ to the $m$th roots of unity may be of interest.
\end{remark}


\subsubsection{Case 2:  $\prod_{i=1}^\ell (\lambda_i^ m - 1) = 0$}
We now consider the case when some of the $\lambda_i$'s are among the $m$th roots of unity. Here we recall (\ref{zd:eq:V*DTF-general}), i.e. if $\lambda_i=\omega^{1-j}$ then the $i$th row of $\begin{pmatrix} \Lambda_1 & \Lambda_2 & \ldots & \Lambda_{m}\end{pmatrix} \DFT$ reads
\begin{displaymath}
( \begin{pmatrix} \Lambda_1 & \Lambda_2 & \ldots & \Lambda_{m}\end{pmatrix} \DFT)_{ij} = \frac{1}{\sqrt{m}} \prod_{\stackrel{k=1}{k\neq j}}^m (\lambda_i - \omega^{1-k}) \omega^{1-j},\;\; ( \begin{pmatrix} \Lambda_1 & \Lambda_2 & \ldots & \Lambda_{m}\end{pmatrix}\DFT)_{ik}=0\;\;\mbox{for}\;k\neq j.
\end{displaymath}
We set $(\mathcal{D}_1)_{ii}=1/\sqrt{m}$, $\Cauchy_{ij}=\prod_{\stackrel{k=1}{k\neq j}}^m (\lambda_i - \omega^{1-k})$ and $\Cauchy_{ik}=0$ for $k\neq j$.
Hence, if e.g. $\lambda_i=\omega^{1-p}$, $\lambda_{i+1}=\omega^{1-q}$ are the only two eigenvalues that match some $m$th root of unity, we have 
the following version of the partition (\ref{eq:C-hat-partition}):
\begin{displaymath}
\widehat{\Cauchy}_p = \begin{pmatrix} {\displaystyle \frac{1}{\lambda_1-\overline{\omega}^{p-1}}}\cr \vdots \cr 
{\displaystyle \frac{1}{\lambda_{i-1}-\overline{\omega}^{p-1}}}\cr
\prod_{\stackrel{k=1}{k\neq p}}^m (\lambda_i - \omega^{1-k}) \cr
 0\cr {\displaystyle \frac{1}{\lambda_{i+2}-\overline{\omega}^{p-1}}}\cr \vdots \cr
 {\displaystyle \frac{1}{\lambda_{\ell}-\overline{\omega}^{p-1}}}
\end{pmatrix},\;\;
\widehat{\Cauchy}_q = \begin{pmatrix} {\displaystyle \frac{1}{\lambda_1-\overline{\omega}^{q-1}}}\cr \vdots \cr 
{\displaystyle \frac{1}{\lambda_{i-1}-\overline{\omega}^{q-1}}}\cr
0\cr
\prod_{\stackrel{k=1}{k\neq q}}^m (\lambda_i - \omega^{1-k}) \cr
{\displaystyle \frac{1}{\lambda_{i+2}-\overline{\omega}^{p-1}}}\cr
\vdots \cr
{\displaystyle \frac{1}{\lambda_{\ell}-\overline{\omega}^{q-1}}}
\end{pmatrix},\;\;
\widehat{\Cauchy}_j = \begin{pmatrix} {\displaystyle \frac{1}{\lambda_1-\overline{\omega}^{j-1}}}\cr \vdots \cr 
{\displaystyle \frac{1}{\lambda_{i-1}-\overline{\omega}^{j-1}}}\cr
0\cr
0 \cr {\displaystyle \frac{1}{\lambda_{i+2}-\overline{\omega}^{p-1}}}\cr
\vdots \cr
{\displaystyle \frac{1}{\lambda_{\ell}-\overline{\omega}^{j-1}}}
\end{pmatrix},\;\;k\not\in\{ p, q\}.
\end{displaymath}
As an illustration of the reconstruction by the product $Z_{\ell}\Delta_{\bfbeta}\widehat{\Cauchy}$ when two of the $\lambda_i$'s are among the roots of unity, consider the following schematic representation of (\ref{eq:Omega-beta-1}) with $m=5$, $\ell=3$:
$$
\underbrace{\left(\begin{smallmatrix}
x & \blacklozenge & \star & x & x \cr
x & \blacklozenge & \star & x & x \cr
x & \blacklozenge & \star & x & x \cr
x & \blacklozenge & \star & x & x \cr
x & \blacklozenge & \star & x & x \cr
x & \blacklozenge & \star & x & x \cr
x & \blacklozenge & \star & x & x \cr
x & \blacklozenge & \star & x & x \cr
\end{smallmatrix}\right)}_{\widehat{\X}_m} \approx
\underbrace{\left(\begin{smallmatrix}
\bullet & \star & \blacklozenge\cr
\bullet & \star & \blacklozenge\cr
\bullet & \star & \blacklozenge\cr
\bullet & \star & \blacklozenge\cr
\bullet & \star & \blacklozenge\cr
\bullet & \star & \blacklozenge\cr
\bullet & \star & \blacklozenge\cr
\bullet & \star & \blacklozenge\cr
\end{smallmatrix}\right)}_{Z_{\ell}}
\overbrace{\begin{pmatrix} \lambda_1 & 0 & 0 \cr 0 & \lambda_2 & 0 \cr 0 & 0 & \lambda_3 \end{pmatrix}}^{\Delta_{\bfbeta}}
\underbrace{\begin{pmatrix}
	+ & + & + & + & +  \cr
	0 & 0 & \star & 0 & 0  \cr
	0 & \blacklozenge & 0 & 0 & 0 
	\end{pmatrix}}_{\begin{pmatrix} \widehat{\Cauchy}_1 & \widehat{\Cauchy}_2 & \ldots & \widehat{\Cauchy}_{m}\end{pmatrix}}
$$
Here $\lambda_2$ and $\lambda_3$ are roots of unity and thus $Z_{\ell}(:,2)$ ($\star$) and $Z_{\ell}(:,3)$ ($\blacklozenge$) participate in the representation of only $\widehat{\X}_m(:,3)$ and $\widehat{\X}_m(:,2)$, respectively.
That is, if $\lambda_i = \omega^{1-j}$, then $Z_{\ell}(:,i)$ participates in the reconstruction of only one transformed snapshot, namely $\widehat{\f}_{j}$.


%% file: SOURCES/gla-connection.tex
\section{Reconstructions based on GLA theory}\label{S=GLA-connection}
We now wish to compare and contrast our above calculations to the Generalize Laplace Analysis (GLA) theorem \cite{Mohr:2014wm,Mezic:2013ei} adapted to this matrix setting. The GLA theorem is a theoretical result for a Koopman operator acting on infinite dimensional function spaces which constructs projections of functions onto Koopman eigenfunctions via ergodic-type averages. The classical formulation begins with merely knowing the spectrum of $\A$ and constructs projections onto the eigenspaces using infinite-time averages. 

\subsection{Matrix GLA}
We call an eigenvalue $\lambda$ of $\A$ a dominant eigenvalue if for all other eigenvalues $\mu$, $\abs{\lambda}\geq\abs{\mu}$. We call $\lambda$ strictly dominant if for all $\mu$, $\abs{\lambda} > \abs{\mu}$.

\begin{proposition}[Matrix GLA]
Let $\lambda\in\sigma(\A)$ be a dominant eigenvalue and $\Pi_{\lambda} : \C^n \to \C^n$ the (skew) projection onto $N(\A - \lambda I)$. Then for any $\f \in \C^n$
	\begin{equation}
	\Pi_{\lambda} \mbf f = \lim_{m\to\infty} \frac{1}{m} \sum_{i=0}^{m-1} \lambda^{-i}\A^i \f.
	\end{equation}
\end{proposition}

\begin{remark}
If we let $\f_i = \A^{i-1}\f$, the above average is
${\displaystyle 	\lim_{m\to\infty} \frac{1}{m} \sum_{i=1}^{m} \lambda^{-i+1} \f_i}$.

\end{remark}
This is easily shown by expanding $\mbf v$ in the eigenvector basis. The above sum is then composed of $n$-terms of the form $\frac{1}{m}\sum_{i=0}^{m-1} \left(\frac{\mu}{\lambda}\right)^{i}$. Since $\lambda$ is a dominant eigenvalue, then for any $\mu \neq \lambda$, the term $\frac{1}{m}\sum_{i=0}^{m-1} \left(\frac{\mu}{\lambda}\right)^{i}$ converges to 0. Projections onto another eigenspace requires subtracting off from $\mbf v$ all projections $\Pi_\lambda \mbf v$ corresponding to all eigenvalues strictly dominating the one of interest and then running the above average for the new vector.

\begin{remark}
The GLA is an inherently infinitary result. A straight forward application using finite data quickly leads to numerical instability.
\end{remark}

The projection operators $\Pi_{\lambda}$ are skew projections along the other eigenvectors of $\A$. We can define a weighted inner product in which these spectral projections become orthogonal. Let $Z_n = [\mbf z_1,\dots, \mbf z_n]$ be the normalized eigenvectors of $\A$ and let $Z_n = QR$ be the QR decomposition of $Z_n$. Let $L = R^{-1}Q^* (\equiv Z_n^{-1} \equiv Z_n^\dag)$ and $P = L^*L$ and define the hermitian form
	\begin{equation}
	\inner{\mbf v}{\mbf w}_{P} = \mbf w^* P \mbf v.
	\end{equation}
Since $Z_n$ is full rank, this form is nondegenerate and defines an inner product. Noting that $L$ is just $Z_n^{-1}$ it can easily be shown that $\inner{\mbf z_i}{\mbf z_j}_{P} = \delta_{i,j}$ which implies that the projections are orthogonal projections with respect to this inner product.

We can formulate a ``coordinate'' version of the GLA theorem. Since we know the eigenvectors, we can reduce the infinitary result to a finitary one.

\begin{proposition}\label{prop:coord-gla}
Let $\f = Z_n \vec{\mbf\alpha} = \sum_{i=1}^{n} z_i \vec{\mbf \alpha}_i$ and define $\f_i = \A^{i-1}\f = Z_n \Lambda_n^{i-1} \vec{\mbf\alpha}$, where $\Lambda_n$ is the diagonal matrix of eigenvalues of $\A$. Then
	\begin{equation}
	\vec{\mbf\alpha} 
	= \frac{1}{m} \sum_{i=1}^{m} \Lambda_n^{-(i-1)} R^{-1}Q^* \mbf f_i 
	= \frac{1}{m} \sum_{i=1}^{m} \Lambda_n^{-(i-1)} L \mbf f_i 
	= \frac{1}{m} \sum_{i=1}^{m} \Lambda_n^{-(i-1)} Z_n^{-1} \mbf f_i . \label{eq:matrix-gla}
	\end{equation}
\end{proposition}

\begin{proof}
Since $Z_n^{-1}$ exists,
	\begin{displaymath}
	\frac{1}{m} \sum_{i=1}^{m}\Lambda_n^{-i+1} Z_n^{-1} \mbf f_i 
	= \frac{1}{m} \sum_{i=1}^{m} \Lambda_n^{-i+1} Z_n^\dag Z_n \Lambda_n^{i-1} \vec{\mbf \alpha} 
	= \frac{1}{m} \sum_{i=1}^{m} \vec{\mbf\alpha} 
	= \vec{\mbf \alpha}.
	\end{displaymath}
By definition $Z_n^{-1} = L = R^{-1}Q^*$. This gives the equivalent formulations.
\end{proof}

\subsection{GLA formulas for reconstruction weights}
As in section \ref{S=Snapshot-rec-theory}, we wish to reconstruct the evolution $(\f_1,\dots, \f_m)$ with a smaller number of Ritz vectors $Z_\ell = [\mbf z_1,\dots, \mbf z_\ell]$, $\ell \leq m$. Given that the matrix version of the GLA (prop.\ \ref{prop:coord-gla}) gives the exact reconstruction weights, we would like to investigate how well GLA reconstructs the weights without full knowledge of the eigenvectors. Since the GLA gives a skew projection onto the eigenfunctions rather than an orthogonal one, the reconstructed weights will not be optimal in minimizing the error's 2-norm.

The objective is to find $\vec{\mbf\alpha} \in \C^{\ell\times 1}$ such that
	\begin{equation}\label{eq:reconstruction-loss-function}
	\vec{\mbf\alpha} = \argmin_{\vec{\mbf\beta} \in \C^{\ell\times 1}} \sum_{i=1}^{m} \norm{\mbf f_i - \sum_{j=1}^{\ell} \mbf z_j \lambda_j^{i-1} \vec{\mbf\beta}_j }_2^2
	\end{equation}
In general, it will not be possible to perfectly reconstruct the data. Define for $\vec{\mbf\beta} = (\vec{\mbf\beta}_1,\dots, \vec{\mbf\beta}_\ell)$, the functionals
$$
	\Omega_i(\mbf\beta) = \norm[\Big]{\mbf f_i - \sum_{j=1}^{\ell} \mbf z_j \lambda_j^{i-1} \vec{\mbf\beta}_j }^2, \;\;
	\Omega(\mbf\beta) = \sum_{i=1}^{m} \Omega_i(\vec{\mbf\beta}).
	$$

First consider the case of a particular $\f_i$. We look for $\vec{\mbf \alpha}^{(i)} \in \C^{\ell\times1}$ that reconstructs $\f_i$ exactly; i.e., the $\vec{\mbf \alpha}^{(i)}$ satsifying
	\begin{equation}
	\vec{\mbf \alpha}^{(i)} = \argmin_{\mbf\beta\in\C^\ell} \Omega_i(\mbf\beta) = \0.
	\end{equation}
Write $Z_\ell = [\mbf z_1 \cdots \mbf z_\ell] \in \C^{n\times \ell}$. The QR-decomposition of this matrix is written as $Z_\ell = \begin{bmatrix}Q_\ell & Q_\ell^\perp\end{bmatrix}\begin{pmatrix} R \\ \mbf 0 \end{pmatrix}$, where $\linspan Z_\ell = \linspan Q_\ell$, and $R \in \C^{\ell\times\ell}$. The pseudo inverse of $Z_\ell$ is
	\begin{equation}
	Z_\ell^\dag = [R^{-1}\, |\, \mbf 0] Q^* .
	\end{equation}
Let $\mbf g_i \in \C^\ell$ and $\mbf h_i \in \C^{n-\ell}$ such that $\mbf f_i = Q\begin{pmatrix} \mbf g_i \\ \mbf h_i \end{pmatrix}$; $\mbf g_i \in \linspan\set{\mbf z_1,\dots,\mbf z_\ell}$ and $\mbf h_i \perp \linspan\set{\mbf z_1,\dots,\mbf z_\ell}$.  Now,
	\begin{align}
	\Omega_i(\mbf \beta) 
	&= \norm{\mbf f_i - Z_\ell \Lambda_\ell^{i-1} \mbf \beta}^2 
	= \norm{ Q^*(\mbf f_i - Z_\ell \Lambda_\ell^{i-1} \mbf \beta )}^2 
	= \norm[\Big]{ \begin{pmatrix} \mbf g_i \\ \mbf h_i \end{pmatrix} - \begin{pmatrix} R \\ \mbf 0 \end{pmatrix} \Lambda_\ell^{i-1} \mbf \beta }^2 \\
	&= \norm{\mbf h_i}^2 + \norm{ \mbf g_i - R \Lambda_\ell^{i-1} \mbf \beta }^2 .
	\end{align}
Thus $\Omega_i(\cdot)$ is minimized with the choice
	$\vec{\mbf\alpha}^{(i)} = \Lambda_{\ell}^{-i+1}R^{-1}\mbf g_i$.
Component-wise this is
	\begin{equation}
	(\vec{\mbf\alpha}^{(i)})_j = \frac{1}{\lambda_j^{i-1}} (R_\ell^{-1}\mbf g_i)_{j}.
	\end{equation}

It is unlikely that $\vec{\mbf\alpha}^{(i)}$ will be the minimizer for every $\Omega_{j}(\cdot)$ and is therefore unlikely to be a minimizer for $\Omega = \Omega_1 + \cdots + \Omega_m$. However, Jensen's inequality will allow us to construct an $\vec{\mbf \alpha}$ which does better than the average of the error's induced by each $\vec{\mbf\alpha}^{(i)}$. Let $\Omega = \set{1,\dots,m}$ and define the probability measure $p:\Omega \to [0,1]$ as $p(i) = \frac{1}{m}$. Define the random vectors $X : \Omega \to \C^n$ as $X(i) = \vec{\mbf\alpha}^{(i)}$. Then since $\Omega$ is a convex function, Jensen's inequality gives us
	\begin{equation}
	\sum_{i=1}^{m} p(i) \Omega( X(i) ) \geq \Omega\Big( \sum_{i=1}^{m} p(i) X(i) \Big) .
	\end{equation}
In other words,
	\begin{equation}
	\frac{1}{m} \sum_{i=1}^{m} \Omega( \vec{\mbf\alpha}^{(i)} ) \geq \Omega \Big( \frac{1}{m}\sum_{i=1}^{m} \vec{\mbf\alpha}^{(i)} \Big).
	\end{equation}
We can rewrite $\frac{1}{m}\sum_{i=1}^{m} \vec{\mbf\alpha}^{(i)}$ as
	\begin{align}
	\frac{1}{m}\sum_{i=1}^{m} \vec{\mbf\alpha}^{(i)}
	&= \frac{1}{m}\sum_{i=1}^{m} \Lambda_{\ell}^{-i+1}R^{-1}\mbf g_i.
	\end{align}
Compare this formula with \eqref{eq:matrix-gla} in Proposition \ref{prop:coord-gla}. This is a GLA formula, except now we can only project onto the $\ell$ eigenfunctions we have. We call the components of this vector the GLA reconstruction weights $\vec{\mbf\alpha}_{(GLA)}$. This can be written in a form similar to $\vec{\mbf\alpha}_\star$ (eq.\ \eqref{eq:optimal-reconstruction-weights-alpha-star}):
	\begin{equation}\label{eq:gla-reconstruction-weights}
	\vec{\mbf\alpha}_{(GLA)} = \frac{1}{m}\sum_{i=1}^{m} \Lambda_{\ell}^{-i+1}R^{-1}\mbf g_i 
	= \begin{pmatrix}
	\frac{1}{m}\sum_{i=1}^{m} \lambda_1^{-i+1}(R^{-1}\mbf g_i)_1 \\
	\frac{1}{m}\sum_{i=1}^{m} \lambda_2^{-i+1} (R^{-1}\mbf g_i)_2 \\
	\vdots \\
	\frac{1}{m}\sum_{i=1}^{m} \lambda_\ell^{-i+1} (R^{-1}\mbf g_i)_\ell \\
	\end{pmatrix}.
	\end{equation}

\subsubsection{Optimal reconstruction formula as a generalized ergodic average.}
Using a reflexive g-inverse, we recall formula \eqref{eq:optimal-reconstruction-weights-alpha-star} giving the optimal reconstruction weights using $\ell$;
	\begin{equation*}
	\vec{\mbf\alpha}_\star = \sum_{i=1}^{m} 
	\begin{pmatrix}
	\frac{\bar{\lambda}_1^{i-1}}{\sum_{k=1}^{m} \abs{\lambda_1}^{2(k-1)}} 	& 	& \\
	&	\ddots & \\
	& 			& \frac{\bar{\lambda}_\ell^{i-1}}{\sum_{k=1}^{m} \abs{\lambda_\ell}^{2(k-1)}}
	\end{pmatrix}
	\begin{pmatrix}
	(R^{-1}\mbf g_i)_1 \\
	\vdots \\
	(R^{-1}\mbf g_i)_\ell
	\end{pmatrix} . \tag{\ref{eq:optimal-reconstruction-weights-alpha-star} revisited}
	\end{equation*}
In the case of unimodular spectrum, the weights given via the reflexive g-inverse reduce exactly to the GLA weights.

\begin{proposition}
When the spectrum of $\A$ lies on the unit circle, $\vec{\mbf\alpha}_{\star} = \vec{\mbf\alpha}_{(GLA)}$.
\end{proposition}

\begin{proof}
For $\lambda_j$ with $\abs{\lambda_j} = 1$, we have $\overline{\lambda_j} = \lambda_j^{-1}$. Then for each $j$,
	\begin{align*}
	(\vec{\mbf\alpha}_{\star})_j = \sum_{i=1}^m\frac{\overline{\lambda}_j^{i-1}}{\sum_{k=1}^m |\lambda_1|^{2(k-1)}} (R^{-1}\g_i)_j 
	&= \sum_{i=1}^m\frac{(\lambda_j^{-1})^{i-1}}{\sum_{k=1}^m 1^{k-1}} (R^{-1}\g_i)_j \\
	&= \sum_{i=1}^m\frac{\lambda_j^{-i+1}}{m} (R^{-1}\g_i)_j 
	= (\vec{\mbf\alpha}_{(GLA)})_j.
	\end{align*}
\end{proof}

When we do not have unimodular spectrum, we can interpret the optimal reconstruction weights as the result of a weighted GLA average. Indeed, the formula for each component $(\vec{\mbf\alpha}_\star)_j$ can be manipulated as follows:
	\begin{align}
	(\vec{\mbf\alpha}_\star)_j 
	&= \sum_{i=1}^{m} \frac{\bar{\lambda}_j^{i-1}}{\sum_{k=1}^{m} \abs{\lambda_j}^{2(k-1)}} (R^{-1}\mbf g_i)_j 
	= \sum_{i=1}^{m} \frac{\bar{\lambda}_j^{i-1}}{\sum_{k=1}^{m} \abs{\lambda_j}^{2(k-1)}} \frac{\lambda_j^{i-1}}{\lambda_j^{i-1}} (R^{-1}\mbf g_i)_j  \\
	&= \sum_{i=1}^{m} \frac{\abs{\lambda_j}^{2(i-1)}}{\sum_{k=1}^{m} \abs{\lambda_j}^{2(k-1)}} \frac{1}{\lambda_j^{i-1}} (R^{-1}\mbf g_i)_j.
	\end{align}
Define
	\begin{equation}
	w_{m,j}(i) = \frac{\abs{\lambda_j}^{2(i-1)}}{\sum_{k=1}^{m} \abs{\lambda_j}^{2(k-1)}}.
	\end{equation}
Then $w_{m,j}(i) > 0$ and 
	${\displaystyle \sum_{i=1}^{m} w_{m,j}(i) = 1}$.
Then, for each component $(\vec{\mbf\alpha}_\star)_j$, we have the (non-uniformly) weighted average
	\begin{equation}\label{eq:optimal-weight-average}
	\vec{\mbf\alpha}_j = \sum_{i=1}^{m} w_{m,j}(i) \lambda_j^{-i+1} (R^{-1}\mbf g_i)_j.
	\end{equation}
Let us define the set of diagonal matrices $W_{m}^{(i)}$ as 
	\begin{equation}
	W_{m}^{(i)} =
	\begin{pmatrix}
	w_{m,1}(i) & & \\
	& \ddots & \\
	& & w_{m,\ell}(i)
	\end{pmatrix} \\
	=\begin{pmatrix}
	\frac{\abs{\lambda_1}^{2(i-1)}}{\sum_{k=1}^{m} \abs{\lambda_1}^{2(k-1)}} & & \\
	& \ddots & \\
	& & \frac{\abs{\lambda_\ell}^{2(i-1)}}{\sum_{k=1}^{m} \abs{\lambda_\ell}^{2(k-1)}}
	\end{pmatrix} .
	\end{equation}
Then, the optimal reconstruction weights are given by the weighted GLA formula
	\begin{equation}\label{eq:optimal-weighted-gla}
	\vec{\mbf\alpha}_\star = \sum_{i=1}^{m} W_{m}^{(i)} \Lambda_\ell^{-i+1} R^{-1} \mbf g_i
	\equiv \sum_{i=1}^{m} W_{m}^{(i)} \Lambda_\ell^{-i+1} R^{-1} \mbf g_i.
	\end{equation}

\begin{remark}
Compare this with the formula for $\vec{\mbf\alpha}_{(GLA)}$ (eq.\ \eqref{eq:gla-reconstruction-weights}), where the weight matrix $W_{m,j}$ for the GLA formula is the uniform weight matrix  $W_m^{(i)} = (1/m) I_m$.
\end{remark}

\subsubsection{Distance between GLA weights and optimal reconstruction weights.}
We have two expressions for reconstructing $m$ snapshots using $\ell$ eigenfunctions, namely \eqref{eq:optimal-reconstruction-weights-alpha-star} and \eqref{eq:gla-reconstruction-weights}. We have already shown that when the spectrum is on the unit circle that these formulas are equivalent. In the case when the spectrum is not contained in the unit circle, the question of their equivalence remains. Clearly, for any fixed number $m$ of snapshots, these formulas give different weights. What about the limit of trying to reconstruct an increasing number of snapshots still only using $\ell$ eigenfunctions? If we write $\vec{\mbf\alpha}_\star^{(m)}$ and $\vec{\mbf\alpha}_{(GLA)}^{(m)}$ for the optimal and GLA weights reconstructing $m$ snapshots, does $\norm{\vec{\mbf\alpha}_\star^{(m)} - \vec{\mbf\alpha}_{(GLA)}^{(m)} } \to 0$ as $m\to \infty$?

In general, the answer is no. We consider the simple case when $\A \in \C^{3 \times 3}$ with eigenvectors $\mbf z_1, \dots, \mbf z_3$, where $\mbf z_1 \perp \mbf z_2$ and $\mbf z_3$ is not orthogonal to either of the other eigenvectors. We also assume that the associated eigenvalues satisfy $1 > \abs{\lambda_1}\geq \abs{\lambda_2} > \abs{\lambda_3}$. The evolution $\f_i = \A^{i-1}\mbf v$ satisfies
	\begin{equation}
	\f_i = \sum_{j=1}^{3} \vec\beta_j \mbf z_j \lambda_j^{i-1}.
	\end{equation}
%
We reconstruct with $\mbf z_1, \mbf z_2$. In this case, the $R^{-1}\mbf g_i$ term in the weights' formulas is
	\begin{equation}
	R^{-1}\mbf g_i = R^{-1} Q_2^* \f_i 
		= \begin{pmatrix}
		\vec\beta_1 \lambda_1^{i-1} + \vec\beta_3 \lambda_3^{i-1} \inner{\mbf z_3}{\mbf z_1} \\
		\vec\beta_2 \lambda_2^{i-1} + \vec\beta_3 \lambda_3^{i-1} \inner{\mbf z_3}{\mbf z_2}
		\end{pmatrix}.
	\end{equation}
For $j=1,2$, define $s_j(i) = \vec\beta_j + \vec\beta_3 (\frac{\lambda_3}{\lambda_j})^{i-1} \inner{\mbf z_3}{\mbf z_j}$ we have
	\begin{displaymath}
	(\vec{\mbf \alpha}_\star^{(m)})_j - (\vec{\mbf \alpha}_{(GLA)}^{(m)})_j 
	= \sum_{i=1}^{m} ( w_{j,m}(i) - \frac{1}{m} ) s_j(i) \\
	= \sum_{i=1}^{m} ( \frac{w_j(i)}{A_m} - \frac{1}{m} ) s_j(i),
	\end{displaymath}
where $w_{j}(i) = \abs{\lambda_j}^{2(i-1)}$ and $A_m = \sum_{i=1}^{m} w_{j}(i)$. Note that $A_m$ is summable
	\begin{equation}
	A_m = \sum_{i=1}^{m} \abs{\lambda_j}^{2(i-1)} = \frac{1 - \abs{\lambda_j}^{2m}}{1 - \abs{\lambda_j}}.
	\end{equation}
We can define the limit of the optimal weights as
	\begin{equation}
	\bar w_j(i) = \lim_{m\to\infty} \frac{w_j(i)}{A_m} = \abs{\lambda_j}^{2(i-1)}(1 - \abs{\lambda_j}).
	\end{equation}
Clearly, since $\bar w_j(i) \to 0$ exponentially fast in $i$, for $m$ large enough, we have that $\bar w_j(1) > m^{-1}$ and $\bar w_j(m) < m^{-1}$. Since $\frac{w_{j}(i)}{A_m} \to \bar w_j(i)$ exponentially fast in $m$, then for all $m$ large enough, $\frac{w_{j}(1)}{A_m} > m^{-1}$ and $\frac{w_{j}(m)}{A_m} < m^{-1}$. Let $t_{m}$ be the largest integer $k$ such that $\bar w_j(k) > m^{-1}$. Then 
	\begin{align}
	(\vec{\mbf \alpha}_\star^{(m)})_j - (\vec{\mbf \alpha}_{(GLA)}^{(m)})_j 
	= \sum_{i=1}^{t_m} \abs*{ \frac{w_j(i)}{A_m} - \frac{1}{m} } s_j(i) - \sum_{i=t_m+1}^{m} \abs*{ \frac{w_j(i)}{A_m} - \frac{1}{m} } s_j(i).
	\end{align}
Since $s_j(i) = \vec\beta_j + \vec\beta_3 (\frac{\lambda_3}{\lambda_j})^{i-1} \inner{\mbf z_3}{\mbf z_j}$ and we can vary $\vec\beta_3$, $\mbf z_3$, and $\inner{\mbf z_3}{\mbf z_j}$ (subject to the above assumptions on the last two), it is clear that we can find signals $\set{s_{j}(\cdot)}$ such that $\lim_{m\to\infty} (\vec{\mbf \alpha}_\star^{(m)})_j - (\vec{\mbf \alpha}_{(GLA)}^{(m)})_j  \not\to 0$ as $m\to \infty$.

What is interesting is that in the limit $\vec{\mbf \alpha}_{(GLA)}^{(m)} \to [\vec\beta_1, \vec\beta_2]^T$; i.e., $\vec{\mbf \alpha}_{(GLA)}^{(m)}$ is asymptotically correct, or a consistent estimator in statistical language. To see this, fix $\eps > 0$ and let  $M \in \N$ be such that $\abs{s_j(i) - \vec\beta_j} < \frac{\eps}{2}$ for $m\geq M$. Then 
	\begin{align}
	\abs*{(\vec{\mbf \alpha}_{(GLA)}^{(m)})_j - \vec\beta_j} 
	&= \abs*{ \Big(\frac{1}{m}\sum_{i=1}^{M} s_j(i) + \frac{1}{m}\sum_{i=M+1}^{m} s_j(i) \Big) -\vec\beta_j} 
	\leq \frac{1}{m} \abs*{ \sum_{i=1}^{M} s_j(i)} + \frac{m-M}{m}\frac{\eps}{2}.
	\end{align}
For all $m$ large enough, we have for each $j=1,2$, $\abs*{(\vec{\mbf \alpha}_{(GLA)}^{(m)})_j - \vec\beta_j} < \eps$.

It is interesting that while each $\vec{\mbf\alpha_\star}^{(m)}$ gives the optimal reconstruction of  $m$ snapshots (i.e., it is an efficient estimator in statistical language), these weights are not asymptotically correct, whereas the GLA weights are sub-optimal for any finite $m$, but are asymptotically correct.

\begin{table}[h]
\begin{center}
\begin{tabular}{|c|c|c|}
\hline
Estimator									& Consistent/Asymptotically Correct 	& Efficient/Optimal \\ \hline
$\vec{\mbf \alpha}_\star^{(m)}$		& no 					& yes \\
$\vec{\mbf \alpha}_{(GLA)}^{(m)}$	& yes 					& no \\ \hline
\end{tabular}
\end{center}
\caption{Comparison of estimators $\vec{\mbf \alpha}_\star^{(m)}$ and $\vec{\mbf \alpha}_{(GLA)}^{(m)}$. Efficiency is with respect to the loss function \eqref{eq:reconstruction-loss-function}. Consistency/Asymptotic correctness is with respect to whether they converge to the correct eigenvector coefficients.}
\label{table:reconstruction-weights-efficiency-consistency}
\end{table}%

\begin{remark}
The reason that GLA weights were asymptotically correct is due to the fact eigenvalues $\lambda_1, \lambda_2$ associated with the reconstruction vectors $\mbf z_1, \mbf z_2$ dominated the eigenvalue of the unresolved eigenvector. This lead the signal $s_j(i)$ to converge exponentially fast to the correct weight $\vec\beta_j$. This result will continue to hold if $\abs{\lambda_3} = \abs{\lambda_2}$. However, the result will be \emph{false}, if any eigenvalue of an unresolved eigenvector has a larger modulus than the reconstruction eigenvalues; i.e., if $\abs{\lambda_2} > \abs{\lambda_j}$ for $j=1$ or 2.
\end{remark}

\subsubsection{Eigenvector-adapted hermitian forms and reflexive g-inverses.}
The weight matrix $M \in \C^{m\ell\times m\ell}$ in Proposition \ref{zd:PROP:M-LS} and the minimizer $\vec{\mbf\alpha}_\star$ can be recovered from a hermitian form on $\C^n$ which is adapted to the eigenvector basis. This gives a connection with the abstract GLA theorem \cite{Mohr:2014wm} in which appropriate spaces of observables for the Koopman operator were constructed by adapting the norm in order to orthogonalize the principal eigenfunctions and their products. 

Let $Z = [\mbf z_1, \dots, \mbf z_n]$ be the eigenvectors of $\A$. These are not necessarily orthogonal with respect to the standard inner product, $\inner{\cdot}{\cdot}_{\C^n}$. We can construct a weighted hermitian form in which the first $\ell$ eigenvectors are orthogonal. Let $Z_\ell = \begin{bmatrix}Q_\ell & Q_\ell^\perp\end{bmatrix}\begin{pmatrix} R \\ \mbf 0 \end{pmatrix}$ be the QR decomposition of $Z_\ell$, where $\Ran Z_\ell = \Ran Q_\ell$, $R \in \C^{\ell\times\ell}$, and $Q = \begin{bmatrix}Q_\ell & Q_\ell^\perp\end{bmatrix}$ is unitary. Define the matrix $P \in \C^{n\times n}$ as
	\begin{align}\label{eq:Cn-weight-matrix}
	P &= Q_\ell R^{-*}R^{-1} Q_\ell^* + Q_\ell^\perp Q_\ell^{\perp *}
	\end{align}
and define the bilinear form $\inner{\cdot}{\cdot}_P : \C^n \times \C^n \to \C$ as
	\begin{equation}
	\inner{\mbf x}{\mbf y}_P = \mbf y^* P \mbf x.
	\end{equation}

\begin{proposition}\label{prop:p-form-properties}
Let $1\leq i,j \leq \ell$ and $\mbf y\in \Ran{Q_\ell^\perp}$. Then
\begin{compactenum}[(a)]
\item $\inner{\mbf z_i}{\mbf z_j}_P = \delta_{i,j}$ and
\item $\inner{\mbf z_i}{\mbf y}_P = 0$.
\end{compactenum}
\end{proposition}

\begin{proof}
For $i \leq \ell$, $\mbf z_i = Z_\ell \vec{\mbf e}_i = Q_\ell R \vec{\mbf e}_i$. Then 
	\begin{align}
	\inner{\mbf z_i}{\mbf z_j}_P = \mbf z_j^* P \mbf z_i = \mbf z_j^* Q_\ell R^{-*}R^{-1}Q_\ell^* \mbf z_i = (R^{-1}Q_\ell^* \mbf z_j)^*(R^{-1}Q_\ell^* \mbf z_i) = (\vec{\mbf e}_j)^*\vec{\mbf e}_i = \delta_{i,j}.
	\end{align}
Write $\mbf y = Q_\ell^\perp \vec c$. Then since $Q_\ell^{\perp *} Q_\ell = \mbf 0$, 
	\begin{align}
	\inner{\mbf z_i}{\mbf y}_P = \mbf y^* P \mbf z_i = (\vec c)^* Q_\ell^{\perp *}(Q_\ell R^{-*}R^{-1} Q_\ell^* + Q_\ell^\perp Q_\ell^{\perp *}) Q_\ell R \vec{\mbf e}_i = 0.
	\end{align}
\end{proof}
\noindent With this inner product, for each $i \leq \ell$, $\Pi_{\mbf z_i}(\cdot) = \inner{\cdot}{\mbf z_i}_P \mbf z_i$ is the \emph{orthogonal} projection onto $\linspan\set{\mbf z_i}$. It is easy to see that $\inner{\cdot}{\cdot}_P$ is a positive, semidefinite hermitian form. It is also nondegenerate; for any $\mbf y \in \C^n$, if $\inner{\mbf y}{\mbf x}_P = 0$ for all $\mbf x$, then $\mbf y=0$. Therefore the hermitian form generates a norm $\norm{\cdot}_P = \sqrt{\inner{\cdot}{\cdot}_P}$. 

\begin{remark}
When the first $\ell$ eigenfunctions are orthonormal, the $P$-norm reduces to the canonical norm on the space. Indeed, if $Z_{\ell} = [\mbf z_1,\dots,\mbf z_\ell]$ are orthonormal, then the QR-decomposition satisfies $Z_\ell = Q \begin{pmatrix} I_\ell \\ \mbf 0 \end{pmatrix}$ where $Q = [Z_\ell ~ Z_{\ell}^{\perp}]$ is unitary. Then for any $\mbf x \in \C^n$
	\begin{align*}
	\norm{x}_{P}^2 &= \inner{\mbf x}{\mbf x}_P 
	= \mbf x^*( Z_\ell I_\ell^{-*}I_\ell^{-1} Z_\ell^* + Z_\ell^\perp Z_\ell^{\perp *})\mbf x \\
	&= \mbf x^*( Z_\ell Z_\ell^* + Z_\ell^\perp Z_\ell^{\perp *})\mbf x 
	= \mbf x^*( Q Q^* )\mbf x 
	= \norm{Q^* x}^2.
	\end{align*}
Since $Q^*$ is unitary, $\norm{Q^* x}^2 = \norm{x}^2$ which gives the result $\norm{x}_{P} = \norm{x}$.
\end{remark}
	
We reformulate \eqref{eq:reconstruction-loss-function} using this $P$-norm.
	\begin{equation}\label{eq:reconstruction-loss-function-P-norm}
	\vec{\mbf\alpha}_P = \argmin_{\vec{\mbf\beta} \in \C^{\ell\times 1}} \sum_{i=1}^{m} \norm{\mbf f_i - \sum_{j=1}^{\ell} \mbf z_j \lambda_j^{i-1} \vec{\mbf\beta}_j }_P^2
	\end{equation}
We write each $\f_i$ as 
	\begin{equation}
	\f_i = \sum_{j=1}^{\ell} \inner{\f_i}{\mbf z_j}_P \mbf z_j + \mbf h_i
	\end{equation}
where $\mbf h_i \perp_P \linspan\set{\mbf z_j \given j=1,\dots, \ell}$. Using Proposition\ \ref{prop:p-form-properties}, 
	\begin{align}
	\sum_{i=1}^{m} \norm{\mbf f_i - \sum_{j=1}^{\ell} \mbf z_j \lambda_j^{i-1} \vec{\mbf\beta}_j }_P^2
	&= \sum_{i=1}^{m} \norm{\mbf h_i}_P^2 +  \norm{\sum_{j=1}^{\ell} \inner{\f_i}{\mbf z_j}_P \mbf z_j - \sum_{j=1}^{\ell} \mbf z_j \lambda_j^{i-1} \vec{\mbf\beta}_j }_P^2 \\
	&= \sum_{i=1}^{m} \norm{\mbf h_i}_P^2 +  \sum_{j=1}^{\ell} \abs{ \inner{\f_i}{\mbf z_j}_P -\lambda_j^{i-1} \vec{\mbf\beta}_j }^2 \\
	&= \sum_{i=1}^{m} \norm{\mbf h_i}_P^2 +  \norm*{ \begin{pmatrix}\inner{\f_i}{\mbf z_1}_P \\ \vdots \\ \inner{\f_i}{\mbf z_\ell}_P \end{pmatrix}  - \begin{pmatrix} \lambda_1^{i-1} & &  \\ &  \ddots & \\ & & \lambda_\ell^{i-1} \end{pmatrix} \vec{\mbf\beta} }_{\C^{\ell}}^2
	\end{align}
Using definition \eqref{eq:Cn-weight-matrix} of $P$ and noting that $R^{-1}Q_\ell^*\mbf z_j = \vec{\mbf e}_j$ and $\mbf z_j^* Q_\ell^\perp = (Q_\ell^{\perp *} \mbf z_j)^* = 0$, we have
$$
	\inner{\f_i}{\mbf z_j}_P 
	= (\mbf z_j)^* (Q_\ell R^{-*}R^{-1} Q_\ell^* + Q_\ell^\perp Q_\ell^{\perp *}) \f_i 
	= (R^{-1}Q_\ell^*\mbf z_j)^* (R^{-1}Q_\ell^* \f_i) 
	= (R^{-1}Q_\ell^* \f_i)_j
	$$
and therefore
	\begin{align}
	\sum_{i=1}^{m} \norm{\mbf f_i - \sum_{j=1}^{\ell} \mbf z_j \lambda_j^{i-1} \vec{\mbf\beta}_j }_P^2
	&= \sum_{i=1}^{m} \norm{\mbf h_i}_P^2 +  \norm*{ R^{-1} Q_\ell^* \f_i  - \Delta_{\Lambda_i} \vec{\mbf\beta} }_{\C^{\ell}}^2 \\
	&= \norm*{ (I_m \otimes R^{-1} Q_\ell^*) \begin{pmatrix}\f_1 \\ \vdots \\ \f_m \end{pmatrix}  -  \begin{pmatrix}\Delta_{\Lambda_1} \\ \vdots \\ \Delta_{\Lambda_m} \end{pmatrix} \vec{\mbf\beta} }_{\C^{m\ell}}^2 +  \sum_{i=1}^{m} \norm{\mbf h_i}_P^2 .
	\end{align}
Defining 
	\begin{equation}
	T = \begin{pmatrix}\Delta_{\Lambda_1} \\ \vdots \\ \Delta_{\Lambda_m} \end{pmatrix},
\;\;\mbox{then}\;\; 
	T^{\dag} = (T^* T)^{-1} T^* = ( \sum_{k=1}^{m} \Delta_{\Lambda_i}^*\Delta_{\Lambda_i})^{-1} \begin{bmatrix}\Delta_{\Lambda_1}^* & \cdots & \Delta_{\Lambda_\ell}^*  \end{bmatrix}.
	\end{equation}
The coefficient vector that solves \eqref{eq:reconstruction-loss-function-P-norm} is
	\begin{equation}
	\vec{\mbf\alpha}_P = T^{\dag}(I_m\otimes R^{-1}Q_{\ell}^*) \begin{pmatrix}\f_1 \\ \vdots \\ \f_m \end{pmatrix}
	\end{equation}
Noting that $Q_\ell^*\f_i = \mbf g_i$ and expanding the above formula for $T^\dag$, we recover formula \eqref{eq:optimal-reconstruction-weights-alpha-star} for the optimal reconstruction weights; i.e., $\vec{\mbf\alpha}_P = \vec{\mbf\alpha}_\star$.

%% file: SOURCES/conclusions.tex
\section{Conclusions}

We have presented a new variant of the Dynamic Mode Decomposition that follows the natural formulation in terms of Krylov bases. Using high accuracy numerical linear algebra techniques we were able to curb the ill-conditioning of the companion matrix's associated Vandermonde matrix allowing us to invert it and find the DMD modes. In addition to the inherent elegance in terms of the companion matrix formulation of DMD, our methods have a close connection to Koopman operator theory, explicitly comparing our methods with the result coming from Generalized Laplace Analysis theory. Furthermore, our methods can be incorporated in a meta-algorithm which reconstructs data snapshots. There exists other formulas for the optimal reconstruction of the snapshots from the DMD modes. Within these algorithms, there is a hierarchy of methods which trade accuracy for faster speed/lower complexity. Our methods can be regarded as the last line of defense; one requires a very accurate result despite very poor condition numbers. We take up this line of enquiry in a companion paper to this one.

%% file: SOURCES/Section_Appendix_1.tex
\section{Appendix: Matlab codes}\label{S=Matlab-codes}

\subsection{Vand\_DFT\_LDU}
\vspace{-2mm}
\begin{algorithm}[H]
	\caption{Matlab implementation of (\ref{zd:eq:Cauchy-LDU})}
	\label{zd:Alg:VAND-FFT-LDU}
	\vspace{-4mm}
	\begin{lstlisting}
	function [XL, D, YU, P1, P2 ] = Vand_DFT_LDU( x, my, LDUorXDY )
	% Vand_FFT_LDU computes entry-wise forward stable LDU decomposition of the matrix
	% G=V(x)*DFT, where V(x)=fliplr(vander(x)) is the Vandermonde matrix defined by
	% the real or complex vector x, and DFT is the Discrete Fourier Transform. V(x) is in
	% general rectangular with [length(x)] rows and [my] columns, V(x)_{ij} = x(i)^(j-1). 
	% The LDU is computed with full pivoting. The code uses explicit formulas for the 
	% Schur complement update. It is written for clarity, and not for optimality.
	% On input:
	% x        :: real or complex vector that defines the Vandermonde matrix V(x).
	% my       :: number of columns of V(x) and the dimension of the DFT matrix
	% LDUorXDY :: job descrption; defines the factors on the output
	%             If 'LDU' , then G(P1,P2) = XL * diag(D) * YU
	%             If 'LU'  , then G(P1,P2) = XL * YU
	%             If 'XDYT', then G = XL * diag(D) * YU'
	%             If 'XYT' , then G = XL * YU' 
	% On exit:
	% XL, YU, D :: The computed factors. XL and YU are matrices and D is column vector
	%              that defines diagonal matrix diag(D). See the descripton of LDUorXDY.                 
	% P1, P2    :: permutations used in the pivoted LDU. See the descripton of LDUorXDY.
	%
	% Coded by Zlatko Drmac, drmac@math.hr.
	%
	mx = max(size(x)) ; y = (exp(-2*pi*1i/my).^(0:my-1)).'  ;
	G  = zeros(mx,my) ; s = 1/sqrt(my) ; tol = sqrt(my)*eps ; 
	for r = 1 : mx, for c = 1 : my
	if ( abs( x(r) - y(c) ) > tol )
	G(r,c) = (s*(x(r)^my - 1)*y(c)) / (x(r)-y(c)) ;
	else G(r,c) = prod(x(r)-y(1:c-1))*prod(x(r)-y(c+1:my)) * y(c) * s ; end
	end; end
	P1 = 1:mx ; P2 = 1:my ; 
	for k = 1 : min(mx,my)   
	[ colmax, jcm] = max( abs( G(k:mx,k:my) ), [] , 1 )        ; 
	[ ~, jm ] = max( colmax ) ; im = jcm(jm)+k-1 ; jm = jm+k-1 ;     
	if ( k ~= im )
	itmp = P1(k)  ; P1(k)  = P1(im)  ; P1(im)  = itmp ;
	tmp  = x(k)   ; x(k)   = x(im)   ; x(im)   = tmp  ;
	vtmp = G(k,:) ; G(k,:) = G(im,:) ; G(im,:) = vtmp ; end
	if ( k~= jm )
	itmp = P2(k)  ; P2(k)  = P2(jm)  ; P2(jm)  = itmp ;
	tmp  = y(k)   ; y(k)   = y(jm)   ; y(jm)   = tmp  ; 
	vtmp = G(:,k) ; G(:,k) = G(:,jm) ; G(:,jm) = vtmp ; end   
	for r = k + 1 : mx, for c = k + 1 : my
	if ( G(r,c) ~= 0 )
	G(r,c) = G(r,c) * (x(r)-x(k))*(y(c)-y(k)) / ((y(c)-x(k))*(x(r)-y(k))) ;
	else G(r,c) = -G(r,k)*G(k,c) / G(k,k) ; end
	end; end    
	end
	D  = diag(G) ; XL = tril(G(1:mx,1:min(mx,my)),-1)*diag(1./D) + eye(mx,min(mx,my)) ;
	if ( strcmp( LDUorXDY, 'LDU') || strcmp( LDUorXDY, 'XDYT') ) 
	YU = diag(1./D)*triu(G(1:mx,1:my),1) + eye(min(mx,my),my) ;
	else YU = triu(G(1:mx,1:my),1) + diag(D)*eye(min(mx,my),my)    ; end
	if ( strcmp( LDUorXDY, 'XDYT' ) || strcmp( LDUorXDY, 'XYT' ) )
	rowpinv(P1) = 1:mx;XL = XL(rowpinv,:); colpinv(P2) = 1:my;YU = YU(:,colpinv)'; end
	end
	\end{lstlisting}
\end{algorithm}

\subsection{X\_inv\_Vandermonde}

\begin{algorithm}[H]
	\caption{Matlab implementation of the formula (\ref{zd:eq:W-via-DFT})}
	\label{zd:Alg:inv(V)-via-DFT}
	\begin{lstlisting}
	function Y = X_inv_Vandermonde( z, X  ) 
	% X_inv_Vandermonde computes Y = X*inv(V(z)), where X has m columns and 
	% V(z)=fliplr(vander(z)) is the m x m Vandermonde matrix defined by the 
	% m x 1 vector z; V(z)_{ij} = z(i)^(j-1), i,j=1,...,m.
	%.......................................................................... 
	% Coded by Zlatko Drmac, drmac@math.hr.
	%..........................................................................
	%
	m = length(z) ; 
	[ L, D, U, p1, p2 ] = Vand_DFT_LDU( z, m, 'LDU' ) ;
	Y = ifft(X,[],2)  ;  
	Y = ( ( Y(:,p2) / U ) * diag(sqrt(m)./D) ) / L ; 
	p1i(p1) = 1:m ; Y = Y(:,p1i) ; % p1i is the inverse of the permutation p1
	end
	\end{lstlisting}
\end{algorithm}